\documentclass[pdflatex,sn-mathphys-num]{sn-jnl}

\usepackage{graphicx}%
\usepackage{multirow}%
\usepackage{amsmath,amssymb,amsfonts}%
\usepackage{amsthm}%
\usepackage{mathrsfs}%
\usepackage[title]{appendix}%
\usepackage{xcolor}%
\usepackage{textcomp}%
\usepackage{manyfoot}%
\usepackage{booktabs}%
\usepackage{algorithm}%
\usepackage{algorithmicx}%
\usepackage{algpseudocode}%
\usepackage{listings}%

\newcommand{\C}{\mathbb{C}}
\newcommand{\R}{\mathbb{R}}

\newcommand{\bz}{\bar{z}}
\newcommand{\bw}{\bar{w}}

\newcommand{\del}{\partial}
\newcommand{\delb}{\bar{\partial}}\newcommand{\dt}{\frac{\partial}{\partial t}}

\newcommand{\gd}{\delta}

\newcommand{\gl}{\lambda}

\newcommand{\ga}{\alpha}
\newcommand{\gb}{\beta}

\renewcommand{\ge}{\varepsilon}

\renewcommand{\bar}[1]{\overline{#1}}

\renewcommand{\i}{\sqrt{-1}}

\newcommand{\bb}{\bar{b}}

\renewcommand{\bf}{\bar{f}}

\newcommand{\bj}{\bar{j}}

\newcommand{\bv}{\bar{v}}

\newcommand{\bnu}{\bar{\nu}}

\newcommand{\HH}{\mathcal{H}}

\DeclareMathOperator{\tr}{tr}

\DeclareMathOperator{\osc}{osc}

\theoremstyle{thmstyleone}%
\newtheorem{theorem}{Theorem}
\newtheorem{proposition}[theorem]{Proposition}%
\newtheorem{corollary}[theorem]{Corollary}
\newtheorem{lemma}[theorem]{Lemma}

\theoremstyle{thmstyletwo}%

\theoremstyle{thmstylethree}%
\newtheorem{definition}{Definition}%

\begin{document}

\title[The parabolic split-type Monge-Amp\`ere on split tangent bundle surfaces]{The parabolic split-type Monge-Amp\`ere on split tangent bundle surfaces}


\author*[1]{\fnm{Joshua} \sur{Jordan}}\email{jojordan@uiowa.edu}

\affil*[1]{\orgdiv{Department of Mathematics}, \orgname{University of Iowa}, \city{Iowa City}, \postcode{52246}, \state{Iowa}, \country{United States}}


\abstract{We introduce a parabolic analogue of the elliptic split-type Monge-Amp\`ere equation developed by Fang and the author, extending Streets' twisted Monge-Amp\`ere equation. The resulting equation is fully nonlinear and non-concave. We prove long-time existence for equations whose exponents are not too far apart and give conditions for convergence to the twisted Monge-Amp\`ere when the exponents approach each other. Applications include long-time convergence on K\"ahler backgrounds and reduction to the twisted Monge-Amp\`ere equation under curvature assumptions.}




\maketitle

\section{Introduction}
Let $M$ denote a compact, complex surface with Hermitian metric $g$. The fundamental $(1,1)$-form $\omega$ is defined by 
$$\omega = \i \sum_{i,j}g_{i\bj}dz^i\wedge d\bz^j.$$
When the holomorphic tangent bundle admits a direct sum decomposition $T^{1,0}M = T^+ \oplus T^-$ in terms of non-trivial holomorphic subbundles $T^\pm$, we will say that $M$ has \textit{split tangent bundle}. When the decomposition is $g$-orthogonal, we will call $g$, and $\omega$, \textit{split-type}. As we will only be concerned with the compact surface case, $\omega$ may be assumed to be pluriclosed, i.e. $\i \del\delb$-closed, by a result of Gauduchon \cite{gauduchon_theoreme_1977}. 

Surfaces admitting a tangent bundle splitting have been classified up to covering by Beauville in dimension 2 \cite{beauville_complex_1998}. All such manifolds are covered either by (1) a product of K\"ahler manifolds, or else (2) a primary, diagonal Hopf surface. Apostolov and Gualtieri refined this to a biholomorphism-type classification in \cite{apostolov_bihermitian_2007}. 

These manifolds have seen an increase in interest in recent years because of their relationship to \textit{generalized complex geometry}, as introduced by Hitchin \cite{hitchin_generalized_2003}. Gualtieri found a natural extension of the K\"ahler condition to generalized complex geometry \cite{gualtieri_generalized_2011}, showing that it corresponds to certain bihermitian geometries first studied in the context of $\mathcal{N}=(2,2)$ supersymmetry \cite{gates_twisted_1978}. To be precise, a bihermitian manifold $(M,I,J,g)$ is \textit{generalized K\"ahler} if 
\begin{equation} \label{eq:H}
I^*d\omega_I = H = -J^*d\omega_J,\quad dH=0.
\end{equation}
Apostolov and Gualtieri then proved that split tangent bundle surfaces are precisely those generalized K\"ahler surfaces with $I\neq \pm J$ whose complex structures commute, i.e. $[I,J]=0$ \cite{apostolov_bihermitian_2007}. 

To better understand compact, complex surface geometry, Streets and Tian introduced \textit{pluriclosed flow} \cite{streets_parabolic_2010}
\begin{equation}\label{eq:pcf}
	\frac{\del \omega}{\del t}  = \frac{\i}{2}\del\delb \log \omega^n +\del\del^*_\omega \omega + \delb\delb^*_\omega \omega.
\end{equation}
And in \cite{streets_generalized_2012}, Streets and Tian noticed that, by allowing the complex structures to flow as well, pluriclosed flow preserves generalized K\"ahler structures. They called this coupling \textit{generalized K\"ahler-Ricci flow}. Fortunately, in the split tangent bundle / commuting-type case, this equation can be gauge-fixed to freeze both complex structures, reducing the generalized K\"ahler-Ricci flow back to (\ref{eq:pcf}) \cite[Section 9.4.3]{garcia-fernandez_generalized_2021}.

It is a well-known fact that K\"ahler-Ricci flow can be reduced to the complex Monge-Amp\`ere equation by the $\i\del\delb$-lemma (see e.g. \cite[Section 0]{yau_ricci_1978-1}). Though generalized K\"ahler manifolds are typically not $\i\del\delb$-manifolds, Gates, Hull, and Ro\v{c}ek showed the existence of local potentials for generalized K\"ahler metrics \cite{gates_twisted_1978}. When these local potentials are global (e.g. when $\dim_{\C}M=2$  \cite[Theorem 1.6]{fang_canonical_2025}), Streets was able to find a scalar reduction for generalized K\"ahler-Ricci flow on split tangent bundle / commuting-type manifolds \cite[Lemma 3.4]{streets_pluriclosed_2016}. The equation is called \textit{complex twisted Monge-Amp\`ere} --- a mixed convex-concave, fully nonlinear, parabolic partial differential equation. He also defined a class of smooth, admissible functions on which the equation is strictly parabolic. Fixing a positive, pluriclosed, split-type $(1,1)$-form $\omega_0 = \omega_0^+ + \omega_0^-$ and taking $u$ to be a real-valued function satisfying 
$$\omega_0 + \i(\del_+\delb_+ - \del_-\delb_-)u > 0,$$
the parabolic twisted Monge-Amp\`ere equation can be written 
\begin{equation}\label{eq:tma}
	\frac{\del u}{\del t} = \log \frac{(\omega_0^+ + \del_+\delb_+ u)^{n_+}}{(\omega_0^+)^{n_+}} - \log \frac{(\omega_0^--\i\del_-\delb_-u)^{n_-}}{(\omega_0^-)^{n_-}}
\end{equation}
where $n_{\pm}=\dim_{\mathbb{C}}T^\pm$. When the splitting is trivial, i.e. $n_- = 0$, commuting-type generalized K\"ahler metrics are automatically K\"ahler and we recover the parabolic Monge-Amp\`ere equation studied by Cao \cite{cao_deformation_1985}.

Since the operator is mixed convex-concave and fails to be a function of the eigenvalues of the Hessian, much of the known elliptic/parabolic theory on compact manifolds cannot be immediately applied to (\ref{eq:tma}). For example, the notion of $\mathcal{C}$-subsolutions due to Guan \cite{guan_second-order_2014} and Szekelyhidi \cite{szekelyhidi_fully_2015} in the elliptic case and Phong \cite{phong_fully_2021} in the parabolic, does not make sense for equations of this type. Additionally, the Evans-Krylov theorem has only begun to be extended to certain classes of non-concave equations in the last couple decades \cite{caffarelli_priori_2000, collins_c2alpha_2016, streets_evans-krylov_2016}.

One of the strongest existence theorems for (\ref{eq:tma}) is due to Streets \cite[Proposition 5.3]{streets_pluriclosed_2018-1}. He proves that (\ref{eq:tma}) can be smoothly solved on a commuting-type generalized K\"ahler surface precisely when a positive lower bound on $\omega_u^+ = \omega_0^++\i\del_+\delb_+ u$ can be found. This result relies crucially on the Evans-Krylov estimate of Streets and Warren \cite{streets_evans-krylov_2016}, where they observe that, the ``formal partial Legendre transform", viewed as a matrix of second derivatives, is a subsolution of the linearized equation.

However, such a lower bound on $\omega_u^+$ is often reliant on the geometry of the splitting (see e.g. \cite[Theorem 7.1]{streets_pluriclosed_2018-1}). This is because the operator is $0$-homogeneous and test functions like those used in \cite{yau_ricci_1978-1,tosatti_complex_2010,gill_convergence_2011-1} for second derivative estimates lose their useful signs. Another situation in which an estimate of this type is proved can be found in work of the author, Garcia-Fernandez, and Streets \cite{garcia-fernandez_non-kahler_2023} (extended by Barbaro \cite{barbaro_bismut_2024}). They draw on techniques from generalized complex geometry \`a la Gualtieri \cite{gualtieri_generalized_2011} to show a correspondence of pluriclosed flow to a coupled Hermitian Yang-Mills flow and then use a convenient choice of background metric and techniques from the study of Hermitian Yang-Mills flow to prove long-time existence.

Attempting to refine the conditions placed on the background geometry, Hao Fang and the author \cite{fang_canonical_2025} instead proceed by breaking the scaling symmetry and considering the \textit{split-type Monge-Amp\`ere equation}, parameterized by $\ga,\gb\in \mathbb{R}$,
\begin{equation}\label{eq:stma}
	\begin{cases}
		\gb\log \frac{\omega_0^+ + \del_+\delb_+ u}{\omega_0^+} - \ga \log \frac{\omega_0^--\i\del_-\delb_-u}{\omega_0^-} = f(x)+b,\\
		\omega_0 + \Box u >0,
	\end{cases}
\end{equation}
where $\gb \neq \ga$ and both $u\in C^\infty(M)$ and $b\in \mathbb{R}$ are unknowns. The authors prove a priori $C^2$ estimates for the family of equations (\ref{eq:stma}) and use the continuity method to prove the existence of smooth solutions. To overcome the lack of concavity, they use Collins' twisted type Evans-Krylov estimate \cite{collins_c2alpha_2016}.

In this paper, we will take a parabolic approach to the split-type Monge-Amp\`ere equation. We prove the following theorem. 
\begin{theorem}\label{thm:main}
	There exists a universal $\gb_0>0$ so that, on a split-tangent surface $M^2$ with pluriclosed, split-type metric $\omega$, for any $\frac{\gb}{\ga} \in (\gb_0,1)$ and any $f\in C^\infty(M)$, the solution $u$ to the flow 
	\begin{equation}\label{eq:fullpstma}
		\begin{cases} 
			\frac{\del u}{\del t} = \beta \log \frac{\omega^+ + \i \del_+\delb_+ u}{\omega^+} - \ga \log\frac{\omega^--\i\del_-\delb_-u}{\omega^-} - f,\\
			u|_{t=0} = u_0\in \mathcal{A}(\omega).
		\end{cases}
	\end{equation}
	exists smoothly and uniquely for all $t\in [0,\infty)$.
\end{theorem}
Two primary difficulties arise in the proof of this result which make it distinct from \cite{fang_canonical_2025}. First, in the elliptic case, it is possible to exchange quantities involving third derivatives of the potential, but, in the parabolic setting,  there is an additional term of the type $\nabla \frac{\del u}{\del t}$. This gives rise to different choices of test functions from those used in \cite{fang_canonical_2025}. The test functions are inspired by work of Streets and Garcia-Fernandez, Streets, and the author in \cite{streets_pluriclosed_2016, garcia-fernandez_non-kahler_2023}. Second, the twisted-type Evans-Krylov theorem due to Collins \cite{collins_c2alpha_2016} is not known to hold in the parabolic case, so we use an Evans-Krylov theorem of Streets and Warren  \cite{streets_evans-krylov_2016}. 

In the process of proving Theorem \ref{thm:main}, we also prove a compactness result for solutions with $\gb\approx 1$.
\begin{theorem}\label{thm:main2}
	Let $(\gb_i)_{i\in\mathbb{N}}$ be a sequence satisfying $\gb_0<\gb_i<1$ and $\gb_i\nearrow 1$. Suppose that $(\omega_{0,i})_{i\in\mathbb{N}}$ is a sequence of pluriclosed, split-type metrics on $M$, $(u_{0,i})_{i\in\mathbb{N}}\subset C^\infty(M)$ functions with $u_{0,i}\in\mathcal{A}(\omega_{0,i})$, and $(u_i)_{i\in\mathbb{N}}\subset C^\infty(M\times [0,\infty))$ satisfying
	$$\begin{cases}
		\frac{\del u}{\del t}_i = \gb_i\log \frac{\omega_{0,i}^++\i\del_+\delb_+ u_i}{\omega_{0,i}^+} - \log \frac{\omega_{0,i}^- - \i\del_-\delb_-u_i}{\omega_{0,i}^-}\\
		u_i|_{t=0}=u_{0,i}\in\mathcal{A}(\omega_{0,i}),
	\end{cases}$$
	with the property that 
	$$\limsup_{i\to\infty}\sup_{M\times [0,\infty)}\operatorname{tr}_{\omega_{u_i}}\omega_{0,i}<\infty,$$
	then there exists a subsequence $(u_{i_j})_{j\in\mathbb{N}}$ converging in $C^\infty_{loc}(M\times [0,\infty))$ to a smooth limit $u$ satisfying
	$$\begin{cases}
		\frac{\del u}{\del t} = \log \frac{\omega_0^++\i\del_+\delb_+ u}{\omega_0^+} - \log \frac{\omega_0^- - \i\del_-\delb_-u}{\omega_0^-},\\
		u|_{t=0}=u_0\in\mathcal{A}(\omega_0).
	\end{cases}$$
	on $M\times [0,\infty)$.
\end{theorem}

These results can be substantially improved by requiring stronger geometric conditions. For example, when $M$ is a K\"ahler product, we can improve Theorem \ref{thm:main}.
\begin{corollary}\label{cor:main}
	If $M^2 = \Sigma_+\times \Sigma_-$ is a product of Riemann surfaces, $\omega_0 = \pi_+^*\omega^+_0 + \pi_-^*\omega_0^-$ is the K\"ahler product metric, and $u_0 = u_0^+ + u_0^-$ for functions $u_0^{\pm}\in C^\infty(\Sigma_\pm)$, then for any $\frac{\gb}{\ga} \in (0,1)$ and any $f = f_+ + f_-$ where $f_\pm \in C^\infty(M)$ satisfying the compatibility condition
	\begin{equation}\label{eq:normalizingF}
	\int_{\Sigma_+} \omega_0^+=\int_{\Sigma_+} e^{\frac{f_+}{\gb}}\omega_0^+,\quad \int_{\Sigma_-} \omega_0^-=\int_{\Sigma_-} e^{-\frac{f_-}{\ga}}\omega_0^-,
	\end{equation}
	the solution $u$ to (\ref{eq:fullpstma}) exists smoothly and uniquely for all $t\in [0,\infty)$ and $\omega_u$ converges exponentially quickly to the steady metric
	$$e^{\frac{f_+}{\gb}}\pi^*_+\omega_0^+ + e^{-\frac{f_-}{\ga}}\pi_-^*\omega_0^-\in [\omega_0]^+.$$
\end{corollary}
Interestingly, $\gb_0 = 0$ in the K\"ahler setting. This suggests that it is likely possible to improve $\gb_0$ for Hermitian metrics either by making topological assumptions, choosing more refined test functions, or choosing a more optimal speed function (c.f. \cite{picard_parabolic_2019}).

For the $\gb$-convergence, a geometric assumption on the curvature of the $T_\pm$ line bundles with respect to $\omega_0$ reduces the necessary condition to one on the initial data instead of the entire flow. This condition is implied when, for example, the background metrics are taken to be K\"ahler.
\begin{corollary}\label{cor:main2}
	Let $(\gb_i)_{i\in\mathbb{N}}$ be a sequence satisfying $\gb_0<\gb_i<1$ and $\gb_i\nearrow 1$. Suppose that $(\omega_{0,i})_{i\in\mathbb{N}}$ is a sequence of pluriclosed, split-type metrics on $M$ with 
	$$\max\left\{\max_M \operatorname{tr}_{\omega_{0,i}^-}F(\omega_{0,i}^+,T^+),\max_M \operatorname{tr}_{\omega_{0,i}^+}F(\omega_{0,i}^-,T^-)\right\}\leq 0,$$ $(u_{0,i})_{i\in\mathbb{N}}\subset C^\infty(M)$ functions with $u_{0,i}\in\mathcal{A}(\omega_{0,i})$ and 
	$$\max_M \tr_{\omega_{u_{0,i}}^+}\omega_{0,i}^+<K$$
	for some $K>0$ independent of $i$, and $(u_i)_{i\in\mathbb{N}}\subset C^\infty(M\times [0,\infty))$ satisfying
	$$\begin{cases}
		\frac{\del u}{\del t}_i = \gb_i\log \frac{\omega_{0,i}^++\i\del_+\delb_+ u_i}{\omega_{0,i}^+} - \log \frac{\omega_{0,i}^- - \i\del_-\delb_-u_i}{\omega_{0,i}^-}\\
		u_i|_{t=0}=u_{0,i}\in\mathcal{A}(\omega_{0,i}),
	\end{cases}$$
	then there exists a subsequence $(u_{i_j})_{j\in\mathbb{N}}$ converging in $C^\infty_{loc}(M\times [0,\infty))$ to a smooth limit $u$ satisfying
	$$\begin{cases}
		\frac{\del u}{\del t} = \log \frac{\omega_0^++\i\del_+\delb_+ u}{\omega_0^+} - \log \frac{\omega_0^- - \i\del_-\delb_-u}{\omega_0^-},\\
		u|_{t=0}=u_0\in\mathcal{A}(\omega_0).
	\end{cases}$$
	on $M\times [0,\infty)$.
\end{corollary}
This corollary justifies our expectation that this flow should give approximate solutions to pluriclosed flow.

To orient the reader, we briefly summarize the sections as they appear. In Section \ref{sec:background}, we will introduce some relevant background material on manifolds with split tangent bundle, set some notation, and introduce some analytic reductions that make the evolution equations cleaner. In Sections \ref{sec:nondegandosc}, \ref{sec:secondorder}, and \ref{sec:holder} we will work out, respectively, estimates for the time derivative, the oscillation of the potential, admissibility along the flow and the second derivatives of the potential, and H\"older regularity for second derivatives of the potential. In Section \ref{sec:mainthms} we prove Theorem \ref{thm:main} and Corollary \ref{cor:main}. In Section \ref{sec:conv} we prove Theorem \ref{thm:main2} and Corollary \ref{cor:main2}. The paper concludes with Appendices \ref{sec:evolonM}, \ref{sec:torsionpot}, and \ref{sec:locpde} in which we collect many of the evolution equations and differential inequalities that are useful for the proofs in Sections \ref{sec:nondegandosc}, \ref{sec:secondorder}, and \ref{sec:holder}, respectively.

\section{Background, Notation, and Analytic Reductions}\label{sec:background}
In this section, we will review some of the basic definitions for split tangent bundle manifolds and admissible functions. For a more detailed treatment, refer to \cite{fang_canonical_2025} (c.f. \cite{garcia-fernandez_generalized_2021}). We will also provide a subsection of analytic reductions (Section \ref{sec:reductions}) which simplify some of the calculations, and a subsection of notation (Section \ref{sec:notation}) to which the reader can refer.
\begin{definition}
	A connected complex manifold $(M^n,I)$ of dimension $n$ with complex structure $I$ is said to have \textit{split tangent bundle} if 
	\begin{equation}\label{eq:split}
		T^{1,0}M = T^+ \oplus T^-
	\end{equation}
	where $T^\pm$ are non-trivial holomorphic sub-bundles.
\end{definition}
In two dimensions, Beauville proved that every such manifold is a quotient of either a product of Riemann surfaces or primary diagonal Hopf surfaces \cite{beauville_complex_1998}. Apostolov and Gualtieri \cite{apostolov_generalized_2007} refined his classification and proved that, on surfaces, split tangent bundle is equivalent to \textit{commuting-type generalized K\"ahler}, allowing us to import some familiar concepts from bihermitian geometry.

The decomposition (\ref{eq:split}) results in decompositions of all tensors, allowing us to define the class of split-type forms.
\begin{definition}
	The space of \textit{one-type $(p,q)$-forms} is 
	$$\Lambda^{p,q}_\pm := \Lambda^{p}(T_\pm^*)\wedge \Lambda^q(\overline{T_\pm^*}).$$
	The space of \textit{split-type $(1,1)$-forms} is 
	$$\Lambda^s(M) : = \{\eta = \eta^+ + \eta^- | \eta^\pm \in \Lambda^{1,1}_\pm\}.$$
	We will call a Hermitian metric $g$ \textit{split-type} if its fundamental form $\omega \in \Lambda^s(M)$.
\end{definition}
Split-type Hermitian metrics are of interest because they have a compatibility with the splitting, namely the splitting $T^{1,0}M = T^+ \oplus T^-$ becomes orthogonal. Additionally, this condition is preserved by multiplication by conformal factors, so Gauduchon's theorem on the existence of standard representative conformal classes holds \cite{gauduchon_theoreme_1977}, i.e. it is no restriction to assume our split-type metrics are pluriclosed.

Additionally, we can decompose the differential $\del = \del_+ + \del_-$ \cite{apostolov_generalized_2007} and define a second order operator $\Box$ on smooth functions by  
$$\Box u = \i (\del_+\delb_+ - \del_-\delb_-)u.$$
Since $\i \del\delb \Box u=0$ for any $u\in C^\infty(M)$, we can define natural classes of pluriclosed, split-type forms parameterized by smooth functions following Streets \cite{streets_pluriclosed_2016} (c.f. \cite{guan_class_2015}), namely
$$[\omega] = \{\omega_u = \omega + \Box u | u\in C^\infty(M)\}, \quad [\omega]^+ = \{\omega_u\in[\omega] | \omega_u>0\}.$$
The interested reader can refer to \cite{fang_canonical_2025} for a proof that the space of all such classes over a compact, complex surface is finite dimensional and a criterion for checking $[\omega]^+\neq\emptyset$.

It will often not be convenient to work directly with the forms in $[\omega]^+$, so we define the cone of functions which parameterize $[\omega]^+$. Such functions are called \textit{admissible}.
\begin{definition}
	$\mathcal{A}(\omega)$ will refer to the \textit{set of admissible functions} for $\omega$
	$$\mathcal{A}(\omega) := \{u\in C^\infty (M)|\omega_u>0\}.$$
\end{definition}

\subsection{Analytic Reductions}\label{sec:reductions}
For $\beta/\ga\in (0,1]$, $f\in C^\infty(M)$, we consider the smooth Cauchy problem for (\ref{eq:fullpstma}) with a fixed, pluriclosed, split-type, reference metric $\omega_0$. In this section, we will make several reductions that simplify the form of the PDE without losing generality of the estimates that follow. 

First, notice that if we rescale all quantities appropriately, the equation can be put into the form
\begin{equation}\label{eq:1reducedpstMA}
	\begin{cases} 
		\frac{\del u}{\del t} = \beta \log \frac{\omega_0^+ + \i \del_+\delb_+ u}{\omega_0^+} - \log\frac{\omega_0^--\i\del_-\delb_-u}{\omega_0^-} - f,\\
		u|_{t=0} = u_0\in \mathcal{A}(\omega_0).
	\end{cases}
\end{equation}
where now $\gb\in (0,1]$.

We can make further reductions in the case of $\gb\in (0,1]$. By applying Theorems 1.11 \& 1.17 of \cite{fang_canonical_2025}, we can gauge out the $f$ term. Namely, there is a smooth metric $\omega_\infty = \omega_0 + \Box u_\infty$ and a real $b$ solving the elliptic split-type Monge-Amp\`ere equation (\ref{eq:stma}) with estimates of all orders depending only on $f$, $\omega_0$, and $\gb$. We can then use (\ref{eq:stma}) to replace $f$ by $0$, $\omega_0$ by $\omega_\infty$, $u$ by $u-u_\infty$, and $u_0$ by $u_0-u_\infty$, so that (\ref{eq:1reducedpstMA}) becomes
\begin{equation}\label{eq:2reducedpstma}
	\begin{cases} 
		\frac{\del u}{\del t} = \beta \log \frac{\omega_0^+ + \i \del_+\delb_+ u}{\omega_0^+} - \log\frac{\omega_0^--\i\del_-\delb_-u}{\omega_0^-} - b,\\
		u|_{t=0} = u_0\in \mathcal{A}(\omega_0).
	\end{cases}
\end{equation}
Notice that the new $u_0$ remains admissible with respect to the new $\omega_0$.

Finally, by adding linear functions to $u$ depending only on time, it will suffice to consider the parabolic split-type Monge-Amp\`ere equation with zero right-hand side and smooth, admissible, non-negative initial data. 
\begin{equation}\label{eq:pstma}
	\begin{cases} 
		\frac{\del u}{\del t} = \beta \log \frac{\omega_0^+ + \i \del_+\delb_+ u}{\omega_0^+} - \log\frac{\omega_0^--\i\del_-\delb_-u}{\omega_0^-},\\
		u|_{t=0} = u_0(x) \in \mathcal{A}(\omega_0),\\
		\min_M u_0 = 0.
	\end{cases}
\end{equation}
By applying our reductions in reverse order, unique long-time existence for (\ref{eq:pstma}) will imply unique long-time existence for (\ref{eq:fullpstma}).

\subsection{Notation}\label{sec:notation}
Some of the following notation will be different for the local equation, which is central to Section \ref{sec:holder}, but these notational changes will be mentioned when they arise. Otherwise, $L$ will refer to the spatial part of the linearized operator of (\ref{eq:pstma}), 
$$L\phi = \frac{\beta}{g\lambda}\phi_{z\bz} + \frac{1}{h\eta}\phi_{w\bw}.$$
We will use $\HH$ to refer to the linearized heat operator
$$\HH = \dt - L.$$

We will always use $\omega_0$ to refer to a pluriclosed, split-type background metric and $\omega_u$ to refer to the pluriclosed $(1,1)$-form
$$\omega_u = \omega_0 + \Box u.$$
Additionally, we will always choose \textit{holomorphic split-type coordinates}, holomorphic coordinates which are compatible with the splitting, i.e. $(z,w)$ so that
$$T^+ = \operatorname{span}\left\{\frac{\del}{\del z}\right\},\quad T^- = \operatorname{span}\left\{\frac{\del}{\del w}\right\}.$$
Such coordinates can be seen to exist by \cite[Theorems 1 \& 4]{apostolov_generalized_2007}. In such coordinates, the form $\omega_0$ will be written 
$$\omega_0 = \i g dz\wedge d\bz + \i h dw\wedge d\bw,$$
so that the pluriclosed condition is 
$$g_{w\bw}+h_{z\bz}=0.$$

Moreover, we will let $\gl$ and $\eta$ be the traces $\operatorname{tr}_{\omega_0^{\pm}}\omega_u^{\pm}$. In coordinates adapted to the splitting, these are written
$$\lambda = 1 + \frac{u_{z\bz}}{g},\quad \eta = 1- \frac{u_{w\bw}}{h}.$$
In two-dimensions, this implies $\omega_u$ can be written
$$\omega_u = \i g\gl dz\wedge d\bz + \i h\eta dw\wedge d\bw.$$

It will also often be convenient to define some, typically non-pluriclosed, $\gb$-adjusted metrics which will turn $L$ into a Chern Laplacian, we will denote this metric by $\tilde{\omega}_u$ and write it in coordinates as
$$\tilde{\omega}_u = \i \frac{g\gl}{\gb} dz\wedge d\bz + \i h\eta dw\wedge d\bw.$$
We will use $\tilde{\cdot}$ to denote any quantities associated to the adjusted metric.

Finally, we define the smooth, admissible existence time $\tau_*$ to be 
$$\tau_* = \sup\{\tau>0| \exists u(\cdot,t)\in \mathcal{A}(\omega_0),\text{ solving (\ref{eq:pstma}) on } M\times[0,\tau)\}.$$
By standard parabolic theory, we know that $\tau_*>0$. The content of Theorem \ref{thm:main} is to show that $\tau_*=\infty$.

\section{Time-Derivative and Oscillation Estimates}\label{sec:nondegandosc}
We begin with two easy estimates which are fundamental to the later discussion.
\begin{proposition}\label{prop:nondeg}
	Let $\gb\in(0,1]$ and $u$ be a solution to (\ref{eq:pstma}) on $M\times [0,\tau)$ with $\tau < \tau_*$, then
	$$\min G_{\gb}(u_0)\leq \frac{\del u}{\del t} \leq \max G_\gb(u_0)$$
	which is equivalent to comparability of the metric coefficients
	$$e^{\min G_\gb(u_0)}\eta \leq \lambda^\beta \leq e^{\max G_\gb(u_0)}\eta$$
	and $G_\gb(u_0)$ is defined over the course of the proof.
	\begin{proof}
		By Lemma \ref{lem:dotucaloric}, we find that 
		$$\HH \frac{\del u}{\del t} = 0.$$
		By the admissibility of the initial data and $\tau_*>0$, we know that
		\begin{equation}\label{eq:dtu} 
			G_\gb(u_0):= \lim_{t\to 0^+}\frac{\del u}{\del t} = \left(\beta \log \frac{\omega_0^+ + \i \del_+\delb_+u_0}{\omega_0^+} - \log \frac{\omega_0^- - \i \del_-\delb_- u_0}{\omega_0^-}\right) \in C^\infty(M).
		\end{equation}
		We will often drop omit the dependence on $\gb$ and $u_0$ when it is understood.
		
		Since, by the maximum principle, $\max_M \frac{\del u}{\del t}(x,t)$ is non-increasing and $\min_M \frac{\del u}{\del t}(x,t)$ is non-decreasing, the result follows.
	\end{proof}
\end{proposition}

The maximum principle then implies the oscillation bound.
\begin{proposition}\label{prop:osc}
	Given the same hypotheses as Proposition \ref{prop:nondeg}, $u$ is bounded on $M\times  [0,\tau)$. Moreover, the bound is independent of $\tau$. Namely,
	$$0 \leq u \leq \max_M u_0.$$
	\begin{proof}
		The result follows immediately from the reductions in Section \ref{sec:reductions} and the maximum principle. 
	\end{proof}
\end{proposition}

\section{Second Order Estimates}\label{sec:secondorder}
We will begin the process of controlling the second derivatives by estimating $\lambda$ and $\eta$ from below, guaranteeing that the solution remains admissible for as long as it exists. This argument proceeds similarly to the proof of the lower bound in \cite{fang_canonical_2025}.

\begin{proposition}\label{prop:lambdalowerbnd}
	With the hypotheses of Proposition \ref{prop:nondeg}, we find that, for every $0 < \gd < \gb$ there is a lower bound of $\gl$ which is independent of $\tau$, namely
	$$\inf_{M\times [0,\tau)}\lambda \geq \min\left\{\left(\gd e^{-\max G}\right)^{\frac{1}{1-\beta}},\frac{\frac{1+(1+\gd)C}{\gb-\gd}}{||\min G|-(1-\gb)|}\right\}e^{-(\frac{1+(1+\gd)C}{\gb-\gd})\max u_0}>0,$$
	where $C = C(\omega_0)\geq 0$ and $G$ is defined in (\ref{eq:dtu}).
	\begin{proof}
		From (\ref{eqn:loglambdaevol}), we find
		\begin{equation}\label{eq:loglambdalowerest}
			\HH\log \lambda \geq \frac{1}{g\lambda}\left(\frac{h_{z\bz}}{h} - \left|\frac{h_z}{h}\right|^2\right) + \frac{1}{h\eta}\left(\frac{g_{w\bw}}{g} - \left|\frac{g_w}{g} \right|^2\right).
		\end{equation}
		
		We consider the test function $\log \lambda + Au$. By (\ref{lem:Hu}), one can see that 
		$$\HH (\log \lambda + A u) \geq A\frac{\del u}{\del t} + \frac{1}{\lambda}(\beta A - C) - \frac{1}{\eta}(A + C) + A(1-\beta)$$
		where 
		$$C = \max\{\max_M|\Omega_{z\bz w}^w(\omega_0)|, \max_M|\Omega_{w\bw z}^z(\omega_0)|\}\geq 0.$$ 
		
		At a space-time minimum $(p,t)\in M\times [0,\tau)$, and making use of Proposition \ref{prop:nondeg}, this becomes 
		\begin{equation}\label{eq:estimate1}
			0 \geq \frac{1}{\lambda}(\beta A - C) -\frac{1}{\eta}(A + C) + A (1 -\beta +\min G).
		\end{equation}
		We then split the argument into two cases based on the size of $\gl$ and $\eta$ at $(p,t)$.
		
		In the first case, suppose that $\gl(p,t) > \gd \eta(p,t)$. Then by Proposition \ref{prop:nondeg}, $\eta \geq \exp(-\max G)\lambda^\beta$, so that 
		$$\lambda(p,t) \geq \gd e^{-\max G}(\lambda(p,t))^\beta.$$
		As $\beta \neq 1$, we find 
		$$\lambda(p,t)\geq \left(\gd e^{-\max G}\right)^{\frac{1}{1-\beta}}.$$ 
		
		In the second case, we suppose that $\gl(p,t)\leq \gd \eta(p,t)$. Plugging this into (\ref{eq:estimate1}), we find
		\begin{equation}
			0 \geq \frac{1}{\lambda}((\gb-\gd)A - (1+\gd)C) - A(|\min G|-(1-\gb)).
		\end{equation}
		By choosing $A = \frac{1+(1+\gd)C}{\gb-\gd}$  we must have $|\min G| \geq (1-\gb)$, otherwise this case is not possible. We also find that
		$$\lambda(p,t) \geq \frac{A}{||\min G|-(1-\gb)|}.$$
		The global estimate is then 
		$$\inf_{M\times [0,\tau)}\lambda \geq \min\left\{\left(\gd e^{-\max G}\right)^{\frac{1}{1-\beta}},\frac{A}{||\min G|-(1-\gb)|}\right\}e^{-A\osc u}$$
		to which one can apply the oscillation bound from Proposition \ref{prop:osc} to find the result.	 
	\end{proof}
\end{proposition}

The limit of the lower bound in Proposition \ref{prop:lambdalowerbnd} as $\gb\nearrow 1$ is very sensitive to the choices of $\gd$ and $G_{\gb}(u_0)$ along the sequence of $\gb$'s. In particular, if one of the factors 
$$\left(\gd e^{-\max G_{\gb}(u_0)}\right)^{\frac{1}{1-\beta}},\quad e^{-(\frac{1+(1+\gd)C}{\gb-\gd})\max u_0}$$
stays positive, then the other converges to zero. 

The bound in Proposition \ref{prop:lambdalowerbnd} is not sharp. If one applies geometric conditions, the bound improves considerable.
\begin{corollary}
	When the curvatures of $T_\pm$ induced by $\omega_0$ satisfy 
	\begin{equation}\label{eq:curvaturecond}
		\max\left\{\max_M \operatorname{tr}_{\omega_0^-}F(\omega_0^+,T^+),\max_M \operatorname{tr}_{\omega_0^+}F(\omega_0^-,T^-)\right\}\leq 0,
	\end{equation} 
	then the estimate can be improved to 
	$$\inf_{M\times [0,\tau)}\gl \geq \min_{M\times \{0\}}\gl>0.$$
	\begin{proof}
		When (\ref{eq:curvaturecond}) holds, the coefficients on the right hand side of (\ref{eq:loglambdalowerest}) are precisely the negatives of the curvatures in the hypothesis.
	\end{proof}
\end{corollary}

This bound is the only place in this discussion where the smooth $\gb\nearrow 1$ limit is obstructed. So, we define $C_0$ by 
\begin{equation}\label{eq:C_0}
	C_0 : = \sup_{M\times [0,\tau_*)}\operatorname{tr}_{\omega_u}\omega_0.
\end{equation}
Writing this using the $\gl$, $\eta$ notation, we see that $C_0$ is determined by the lower bound on the second derivatives
$$C_0 = \sup_{M\times [0,\tau_*)}\left(\frac{1}{\gl}+\frac{1}{\eta}\right).$$
By Propositions \ref{prop:nondeg} and \ref{prop:lambdalowerbnd}, $0< C_0 < \infty$ whenever $\gb \neq 1$ and $u_0$ is admissible.

\subsection{Mixed Derivative Bound}
To estimate the norm squared of the mixed derivative, we will consider the test function 
$$\Psi = |\del_+\del_-u|_{\tilde{\omega}}^2 + A\left(\frac{1}{\lambda} +\frac{1}{\eta}\right) .$$
For more details regarding the evolution of $|\del_+\del_-u|^2_{\tilde{\omega}}$, see Appendix \ref{sec:torsionpot}. The test function $\Psi$ is motivated by Streets' work in \cite[Theorem 1.8]{streets_pluriclosed_2016}, but Streets bounded the torsion and metric simultaneously. We break these estimates up into Propositions \ref{prop:torsionbnd} and \ref{prop:uprbnd} to better understand dependencies on $\gb$.	

\begin{proposition}\label{prop:Psievol}
	At any point $(p,t)\in M\times [0,\tau)$ with $\tau<\tau_*$ which satisfies $|\del_+\del_-u|_{\tilde{\omega}}(p,t)> 1$, the differential inequality
	\begin{equation}\label{eq:Psievol}
		\HH\Psi \leq C_{11}\Psi,
	\end{equation}
	holds with $A, C_{11}\geq 0$, which  are defined in the course of this proof. Moreover, when $\omega_0$ is K\"ahler, then $A=0$ and $C_{11}=2$.
	\begin{proof}
		We observe that that $|\del_+\del_-u|_{\tilde{\omega}}(p,t)^2\geq |\del_+\del_-u|_{\tilde{\omega}}(p,t)>1$ and apply Corollary \ref{cor:lambdacompositions} and Lemma \ref{lem:Hnunormsquared} with $\ge=\gb$, $\gd = \frac{1}{2}$ to obtain 
		\begin{align*}
			\HH \Psi \leq &\ C_8\left(\frac{1}{\gb}+2-\beta^2\right) + C_3|\del_+\del_-u|_{\tilde{\omega}} + C_6|\del_+\del_-u|^2_{\tilde{\omega}} + (C_7-\gb A)\left(\frac{1}{h\eta^2}\left|\frac{\eta_w}{\eta}\right|^2 + \frac{1}{g\lambda^2}\left|\frac{\lambda_z}{\lambda}\right|^2\right)\\
			&\  -A \frac{1}{g\lambda^2}\left(\frac{h_{z\bar{z}}}{h} -\left|\frac{h_z}{h}\right|^2\right) - A\frac{1}{h\lambda\eta}\left(\frac{g_{w\bar{w}}}{g} - \left|\frac{g_w}{g}\right|^2  \right)\\
			&\ - A\frac{\beta}{g\lambda\eta}\left(\frac{h_{z\bz}}{h} -\left|\frac{h_z}{h}\right|^2\right) -  A\frac{\beta}{h\eta^2}\left(\frac{g_{w\bw}}{g} - \left|\frac{g_w}{g}\right|^2\right).
		\end{align*}
		The constant $C_0=\max_M(\frac{1}{\gl}+\frac{1}{\eta})$ is defined in (\ref{eq:C_0}), and we define $$C_9 = C_0^2\max \{|\Omega_{w\bw z}^z(\omega_0)|,|\Omega_{z\bz w}^w(\omega_0)|\},$$ and $C_{10} = C_8\left(\frac{1}{\gb}+2-\beta^2\right)$. Then
		\begin{equation}
			\HH \Psi \leq (C_{10}+C_9A) + C_3|\del_+\del_-u|_{\tilde{\omega}} + C_6|\del_+\del_-u|^2_{\tilde{\omega}} + (C_7-\gb A)\left(\frac{1}{h\eta^2}\left|\frac{\eta_w}{\eta}\right|^2 + \frac{1}{g\lambda^2}\left|\frac{\lambda_z}{\lambda}\right|^2\right).
		\end{equation}
		
		By taking $A = C_7/\gb$ and $C_{11} = C_{10}+C_9A + C_3 +C_6$, we have proved the result.
	\end{proof}
\end{proposition}

Applying the maximum principle will yield the following estimate.
\begin{proposition}\label{prop:torsionbnd}
	There are constants depending only on $\gb$ and background data so that for any $\gb \in (0,1)$, the mixed second derivative is bounded
	$$\max_{M\times \{t\}}|\del_+\del_-u|^2_{\tilde{\omega}} \leq \max\left\{1+ C_0A,(\sup_{M\times \{0\}}|\del_+\del_-u|^2_{\tilde{\omega}} + C_0 A) e^{C_{11}t}\right\}$$
	where $A, C_{11} \geq 0$  are defined in Proposition \ref{prop:Psievol}. Moreover, when $\omega_0$ is K\"ahler, then $A=0$ and $C_{11}=2$.
	\begin{proof}
		Consider a point $(p,t)\in M\times [0,\tau)$ so that
		$$\max_{M\times \{t\}}\Psi = \Psi(p,t).$$
		At such a point, either $|\del_+\del_-u|_{\tilde{\omega}}(p,t)\leq 1$ or $|\del_+\del_-u|_{\tilde{\omega}}(p,t)< 1$. If $|\del_+\del_-u|_{\tilde{\omega}}(p,t)\leq 1$, then, by Proposition \ref{prop:lambdalowerbnd}, we have the estimate 
		$$\max_{M\times \{t\}}|\del_+\del_-u|_{\tilde{\omega}}\leq \max_{M\times \{t\}}\Psi = \Psi(p,t) \leq 1+C_0A.$$ 
		
		Otherwise, $|\del_+\del_-u|_{\tilde{\omega}}(p,t)\geq 1$ and Proposition \ref{prop:Psievol} and the maximum principle implies 
		$$\max_{M\times \{t\}}|\del_+\del_-u|_{\tilde{\omega}}^2\leq \max_{M\times\{t\}} \Psi \leq (\max_{M\times \{0\}}\Psi)e^{C_{11}t}.$$
		and the result follows.
	\end{proof}
\end{proposition}

\subsection{The Upper Bound}
We will now consider the test function 
$$\Phi = \log \lambda + A\left(\frac{1}{\lambda} +\frac{1}{\eta}\right) + B |\del_+\del_-u|_{\tilde{\omega}}^2.$$
For our argument to work we will need to require $\gb>\gb_0$ to make some uncontrolled third order terms negative. The choice is universal (namely $\gb_0 = \frac{2\sqrt{3}-3}{3}$) and seems to be optimal for this choice of test function, but can likely be improved (see e.g. Section \ref{sec:conv}).

\begin{proposition}\label{prop:Phievol} Letting $\Phi = \log \lambda +A\left(\frac{1}{\lambda} +\frac{1}{\eta}\right) + B|\del_+\del_-u|^2_{\tilde{\omega}}$ and $\beta\in (\gb_0,1)$, then
	\begin{equation}\label{eq:HHPhi}
		\HH\Phi \leq C_{14}\Phi.
	\end{equation}
	In the above, $\gb_0 = \frac{2\sqrt{3}-3}{3}$, $C_{14}$ is defined over the course of this proof, and
	$$A = \frac{BC_7}{\gb},\quad B= \frac{8(1+\gb)}{\gb(3\gb^2 + 6\gb - 1)},$$
	where  $C_7$ is defined in Lemma \ref{lem:Hnunormsquared}.
	Moreover, when $\omega_0$ is K\"ahler, $A=0$ and $C_{14} = 2B$.
	\begin{proof}
		Starting with Corollary \ref{cor:lambdacompositions} and Lemma \ref{lem:Hnunormsquared} (with $\ge>0$ and $\gd\in (0,1)$ to be determined), we  drop several negative terms which are not of a useful order in $\lambda$, and define $C_{12} = C_8\left(\frac{1}{\ge} + \frac{1}{\gd}-\beta^2\right)$, $C_{13} = C_0^2\max\{\Omega_{w\bw z}^z(\omega_0),\Omega_{z\bz w}^w(\omega_0)\}$ to obtain
		\begin{align}
			\HH \Phi \leq &\ C_{12}B + C_{13}A + BC_3|\del_+\del_-u|_{\tilde{\omega}} + BC_6|\del_+\del_-u|^2_{\tilde{\omega}} + (BC_7 - \gb  A)\left(\frac{1}{h\eta^2}\left|\frac{\eta_w}{\eta}\right|^2 + \frac{1}{g\lambda^2}\left|\frac{\lambda_z}{\lambda}\right|^2\right)\label{eq:5.6.1}\\
			&\ +\left( 1 + B[-\beta^2(1-\gd) + \sqrt{\beta}(1-\beta)|\del_+\del_-u|_{\tilde{\omega}} - (1+\gb-\ge) |\del_+\del_-u|^2_{\tilde{\omega}}]\right)\nonumber\\
			&\ \times \left(\frac{1}{h\eta}\left|\frac{\lambda_w}{\lambda} + \frac{g_w}{g}\right|^2 + \frac{1}{g\lambda} \left|\frac{\eta_z}{\eta} + \frac{h_z}{h}\right|^2\right).\nonumber
		\end{align}
		
		We will now analyze the coefficient of the uncontrolled third order term. It is given by the polynomial
		$$p_{\ge,\gd}(x) = -\beta^2(1-\gd) + \sqrt{\beta}(1-\beta)x - (1+\gb-\ge)x^2.$$ 
		We claim that $\gb > \gb_0 =\frac{2\sqrt{3}-3}{3}$ implies that there exists $\ge >0$ and $\gd \in (0,1)$ so that
		$$\max p_{\ge,\gd}(x) < 0.$$
		To see this, notice that, whenever $0<\ge<1$, the maximum (as a function of $\ge$ and $\gd$) can be written 
		$$P(\ge,\gd) := \max p_{\ge,\gd}(x) = -\gb^2(1-\gd) + \frac{\gb(1-\gb)^2}{4(1+\gb-\ge)}.$$
		The function $P:(0,1)\times (0,1)\to \R$ is clearly continuous and satisfies 
		$$\lim_{(\ge,\gd)\to (1,1)}P = \frac{(1-\gb)^2}{4}>0> \frac{\gb(1-6\gb - 3\gb^2)}{4(1+\gb)} = \lim_{(\ge,\gd)\to 0^+}P(\ge,\gd).$$
		By continuity and the connectedness of the domain, there exists $\ge \in (0,1)$ and $\gd\in(0,1)$ for which
		$$P(\ge,\gd) = \frac{\gb(1-6\gb - 3\gb^2)}{8(1+\gb)} <0$$
		as $\gb_0 < \gb < 1$.
		
		Choosing $\ge$ and $\gd$ as above and taking $B$ to be
		$$B = \frac{8(1+\gb)}{\gb(3\gb^2 + 6\gb - 1)}>0,$$
		the $\lambda_w$ and $\eta_z$ terms drop out. We can then choose $A = BC_7/\gb$ to make the $\lambda_z$ and $\eta_w$ terms to drop out. In conclusion, we obtain 
		\begin{equation*}
			\HH \Phi \leq C_{12}B + C_{13}A + BC_3|\del_+\del_-u|_{\tilde{\omega}} + BC_6|\del_+\del_-u|^2_{\tilde{\omega}}\leq (C_{12}B + C_{13}A + \frac{BC_3}{2})+B(C_6+\frac{C_3}{2})|\del_+\del_-u|^2_{\tilde{\omega}}.
		\end{equation*}
		Setting $$C_{14} = (C_{12}B + C_{13}A + \frac{BC_3}{2})+B(C_6+\frac{C_3}{2}),$$
		the result follows.
	\end{proof}
\end{proposition}

\begin{proposition}\label{prop:uprbnd}
	With the same setup as Proposition \ref{prop:Phievol}, we find
	$$\max_{M\times \{t\}} \gl \leq (\max_{M\times \{0\}}\lambda)e^{\left((B\max_{M\times\{0\}}|\del_+\del_-u|^2_{\tilde{\omega}} + \frac{BC_7C_0}{\gb})e^{C_{14}t}\right)}.$$
	\begin{proof}
		Proposition \ref{prop:Phievol} implies, by the maximum principle, that 
		$$\max_{M\times \{t\}} \Phi \leq (\max_{M\times \{0\}}\Phi)e^{C_{14}t}.$$
		However, notice that 
		$$\log \max_{M\times \{t\}} \lambda=\max_{M\times \{t\}}\log \lambda \leq \max_{M\times \{t\}}\Phi$$
		and the result follows.
	\end{proof}
\end{proposition}

\section{Estimates for the Modulus of Continuity of the Second Derivatives}\label{sec:holder}
Before we state the estimate, we mention some notation and review Streets-Warrens Evans-Krylov theorem \cite{streets_evans-krylov_2016}. In what follows, $U\subset \C^2$,  $F:\mathrm{Herm}_{+}^2 (U) \to \mathbb{R}$, $\HH = (\dt - L)$ is the linearized operator, and $W$ is the \textit{formal partial Legendre transform}
\begin{equation}\label{eq:formalpartialleg}
	W := \begin{pmatrix}
		u_{z\bz} - u_{z\bw}u^{\bw w}u_{w\bz} & u^{w\bw}u_{w\bz}\\
		u_{z \bw}u^{\bw w} & -u^{w \bw}
	\end{pmatrix}.
\end{equation}
Finally, the natural basis for the parabolic topology is
$$Q((x,t),R)= \{(y,s)\in\C^2\times \R : s\leq t,\, \max\{|y-x|,\sqrt{|t-s|}\}<R\}.$$
\begin{theorem}[{\cite[Theorem 4.3]{streets_evans-krylov_2016}}]\label{thm:SW}
	Suppose that $u\in C^4(Q(R))$ with estimates 
	$$0<A^{-1}\leq \Box u \leq A$$
	on $Q(R)$ satisfying 
	$$\frac{\del u}{\del t} = F(W(\i \del\delb u))$$
	for $F$ a $(A^{-1},A)$-elliptic functional and suppose that $W$ and $\frac{\del u}{\del t}$ satisfy 
	$$\begin{cases}
		\HH\frac{\del u}{\del t} = 0\\
		\HH W(v,\bv)\leq 0
	\end{cases}$$
	for any $v\in \C^2$. Then there are positive constants $\mu$ and $C$ depending on $n$ and $A$ such that for all $r<R$,
	$$\osc_{Q(r)} \frac{\del u}{\del t} + \osc_{Q(r)}W\leq C\left(\frac{r}{R}\right)^{\mu}\left(\osc_{Q(R)} \frac{\del u}{\del t} + \osc_{Q(R)}W\right).$$
\end{theorem}

We now seek to apply this result to the local equation
$$\frac{\del u}{\del t} = \gb \log (u_{z\bz}) - \log (-u_{w\bw}).$$ 
Before writing the proof, we will make some notational simplifications. 

First, we redefine $\gl$ and $\eta$ for the local equation. Namely, 
$$\gl = u_{z\bz},\quad \eta=-u_{w\bw}.$$
With this notational convention, the local version of (\ref{eq:pstma}) will look quite familiar,
\begin{equation}\label{eq:flatpstMA}
	\frac{\del u}{\del t} = \gb\log \gl - \log \eta.
\end{equation}
And the linearized operator is 
$$L\phi = \frac{\gb}{\gl}\phi_{z\bz} + \frac{1}{\eta}\phi_{w\bw}.$$ 

\begin{proposition}\label{prop:Holderest}
	Suppose that $u\in C^4(Q(R))$ with an estimate
	$$0<A^{-1} \leq \Box u\leq A$$
	satisfying (\ref{eq:flatpstMA}), then 
	$$\max\{[\frac{\del u}{\del t}]_{C^{\ga}(Q(\frac{R}{2}))},[\i \del\delb u]_{C^\ga(Q(\frac{R}{2}))}\}\leq C.$$
	\begin{proof}
		The equation (\ref{eq:pstma}) is of the type covered by Theorem \ref{thm:SW}. To see this, one can rewrite (\ref{eq:flatpstMA}) as
		\begin{equation}\label{eq:pstMAWform}
			\frac{\del u}{\del t} = \beta\log \det W + (1-\gb)\log (W_{w\bw}).
		\end{equation}
		
		The fact that the evolutions
		$$\begin{cases}
			\HH \frac{\del u}{\del t} =0\\
			\HH W \leq 0
		\end{cases}$$
		hold is a computation which can be found in Lemma \ref{lem:ek}. Theorem \ref{thm:SW} therefore immediately yields oscillation decay. This leads to a H\"older estimate in the standard way.
	\end{proof}
\end{proposition}

\section{Proofs of Theorem \ref{thm:main} and Corollary \ref{cor:main}}\label{sec:mainthms}
To prove long time existence, we apply a blow-up argument following \cite[Theorem 1.2]{streets_evans-krylov_2016}.
\begin{proof}[Proof of Theorem \ref{thm:main}]
	We begin by assuming towards a contradiction that $\tau_* < \infty$ so that 
	$$\limsup_{t\to \tau_*}\|u\|_{C^{3}(M)}(t) = \infty.$$
	
	It is then possible to pick points $(p_i,t_i) \in M\times [0,\tau_*)$ so that $t_i\nearrow \tau_*$ and 
	$$|\nabla^3 u|_{\omega_0}(p_i,t_i) = \max_{M\times [0,t_i]} |\nabla^3u|_{\omega_0} =:\mu_i \nearrow \infty.$$
	By compactness of $M$, it is possible to choose $(p_i)_{i\in \mathbb{N}}$ so that $p_i\to p$. Therefore, we may pick a split-type coordinate chart centered on $p\in M$ and pull everything back to a subset of $\C^2$. In such a chart $u$ is seen to solve 
	$$\frac{\del u}{\del t} = \gb\log \left(g+u_{z\bz}\right) - \log\left(h-u_{w\bw}\right) + f$$
	where $f = \log \frac{h}{g^\gb}$ is smooth. 
	
	We can then rescale $u$, $g$, $h$, and $f$ by $\ell_i := \ge\mu_i$ for some $\ge>0$ to be determined later. 
	\begin{align*}
		\hat{u}_i(x,t) =&\ \ell_i^{2}\left(u(\ell_i^{-1} x+x_i, \ell_i^{-2}t+t_i) - u(x_i,t_i)\right),\\
		\hat{g}_i(x,t) = &\ g(\ell_i^{-1}x+x_i,\ell_i^{-2}t+t_i),\\
		\hat{h}_i(x,t) = &\ h(\ell_i^{-1}x+x_i,\ell_i^{-2}t+t_i),\\
		\hat{\omega}_i(x,t) =&\ \omega_0(\ell_i^{-1}x+x_i,\ell_i^{-2}t+t_i),\\
		\hat{f}_i(x,t) =&\ F(\ell_i^{-1}x+x_i,\ell_i^{-2}t+t_i),
	\end{align*}
	where $x \in \C^2$ and $x_i=(z_i,w_i)\mapsto p_i\in M$ under the coordinate mapping.	The functions $\hat{u}_i$ all satisfy split-type Monge-Amp\`ere equations on a parabolic cylinder $Q((0,0),3)$ when $i$ is sufficiently large. To be precise, they solve
	$$\dt\hat{u}_i = \gb \log (\hat{g}_i + (\hat{u}_i)_{z\bz}) - \log (\hat{h}_i-(\hat{u}_i)_{w\bw}) + \hat{f}_i.$$
	By construction, $\hat{u}_i$ satisfies 
	$$\sup_{Q((0,0),3)}|\nabla^3 \hat{u}_i|_{\hat{\omega}_i}(t_i) = \frac{1}{\ell_i}\sup_{Q((0,0),3)}|\nabla^3u|_{\omega_0}(\ell_i^{-2}t)\leq \frac{1}{\ge}$$
	for all $i$.
	
	The scalings were specifically chosen to preserve the second derivatives of $u$, so the operators remains uniformly parabolic by Propositions \ref{prop:lambdalowerbnd} and \ref{prop:uprbnd} with parabolicity constants independent of $i$. Therefore, we can apply the interior Schauder estimates \cite[Theorem 6.2]{gilbarg_elliptic_2001} to obtain third derivative estimates $\|\tilde{u}_i\|_{C^3(Q((0,0),2))}\leq C(\ge)$ which are uniform in $i$. By Ascoli-Arzela, we can then extract a subsequence, which we still call $(\hat{u}_i)_{i\in \mathbb{N}}$, and which converges in $C^{2,\ga}$ to a function $\hat{u}_\infty$. Additionally, as $\hat{f}_i$, $\hat{g}_i$, and $\hat{h}_i$ are all smooth, they converge smoothly to constants. Thus, in $Q((0,0),2)$, $\hat{u}_\infty$ will classically solve 
	$$\dt \hat{u}_\infty = \gb \log (\hat{g}_\infty + \hat{u}_\infty)_{z\bz} - \log (\hat{h}_\infty - (\hat{u}_\infty)_{w\bw}) + \hat{f}_\infty.$$
	Replacing $\hat{u}_\infty$ by $\hat{u}_\infty-t\hat{f}_\infty + \hat{g}_\infty |z|^2 - \hat{h}_\infty|w|^2$, we find that $\hat{u}_\infty$ solves 
	$$\dt \hat{u}_\infty = \gb \log (\hat{u}_\infty)_{z\bz} - \log ( - (\hat{u}_\infty)_{w\bw}).$$
	
	By Proposition \ref{prop:Holderest}, we know that $\|\hat{u}_\infty\|_{C^{1,\ga}_t\cap C^{2,\ga}_x(Q(0,0),1)} < C$  and by the interior Schauder estimates \cite[Theorem 6.2]{gilbarg_elliptic_2001}, we get 
	$\|\hat{u}_\infty\|_{C^3(Q((0,0),\frac{1}{2}))}<C$. However,
	$$|\nabla^3 \hat{u}_\infty|_{\hat{\omega}_\infty}(0,0) = \lim_i |\nabla^3\hat{u}_i|_{\hat{\omega}_i}(0,0) = \lim_i \ell_i^{-1}|\nabla^3u|_{\omega_0}(x_i,t_i) = \ge^{-1}.$$
	So by choosing $\ge$ sufficiently small, we get a contradiction to Proposition \ref{prop:Holderest} by way of the Schauder estimates.
\end{proof}

The Corollary for K\"ahler surfaces follows easily from Proposition \ref{prop:torsionbnd}.
\begin{proof}[Proof of Corollary \ref{cor:main}]
	We start with the flow
	\begin{equation}\label{eq:KreducedpstMA}
		\begin{cases} 
			\frac{\del u}{\del t} = \beta \log \frac{\pi_+^*\omega_0^+ + \i \del_+\delb_+ u}{\pi_+^*\omega_0^+} - \ga\log\frac{\pi_-^*\omega_0^--\i\del_-\delb_-u}{\pi_-^*\omega_0^-} - f_+ - f_-,\\
			u|_{t=0} = u_0^+ + u_0^-\in \mathcal{A}(\omega_0).
		\end{cases}
	\end{equation}
	Then, following the analytic reduction process in Section \ref{sec:background}, we will change background metrics to the K\"ahler metric 
	$$\mu = \pi_+^*(e^{\frac{f_+}{\gb}}\omega_0^+) + \pi_-^*(e^{-\frac{f_-}{\ga}}\omega_0^-).$$
	
	Recall, the bracket operation from \cite{fang_canonical_2025}
	$$\{\omega_1,\omega_2\} := \int_M \omega_1^+ \wedge \omega_2^- - \int_M\omega_2^+\wedge \omega_1^-.$$
	In Theorem 3.5, they show that whenever $[\omega_1]^+\neq \emptyset$ this bracket vanishes iff $[\omega_1]=c[\omega_2]$ for some $c\in \R$ (refer to Section \ref{sec:background} for definitions). 
	
	Our choice of normalization (\ref{eq:normalizingF}) guarantees that 
	$$\{\omega_0,\mu\}=0$$
	and further that $c=1$. Therefore, $\mu\in [\omega_0]^+$ and is K\"ahler. We now choose $u_\infty = u_\infty^+ + u_\infty^- \in \mathcal{A}(\omega_0)$ so that $u_\infty^\pm\in\Sigma_\pm$ and
	$$\mu = \omega_0 + \Box u_\infty.$$
	This can be done by applying the $\i\del\delb$-lemma on each factor $\Sigma_{\pm}$. The closedness of $\mu^\pm - \omega_0^\pm$ on $\Sigma_\pm$ is obvious, and the exactness follows from the normalization condition and dimension considerations.
	
	The analytic reduction process in Section \ref{sec:background} yields the equation
	\begin{equation}\label{eq:KreducedpstMA1}
		\begin{cases} 
			\frac{\del}{\del t}\ga^{-1}(u-u_\infty) = \frac{\beta}{\ga} \log \frac{\ga^{-1}\mu^+ + \i \del_+\delb_+ (\ga^{-1}(u-u_\infty^+))}{\ga^{-1}\mu^+} - \ga\log\frac{\ga^{-1}\mu^--\i\del_-\delb_-(\ga^{-1}(u-u_\infty^-))}{\ga^{-1}\mu^-},\\
			\ga^{-1}(u-u_\infty)|_{t=0} = \ga^{-1}(u_0^+ - u_\infty^+) + \ga^{-1}(u_0^- - u_\infty^-)\in \mathcal{A}(\ga{^-1}\mu).
		\end{cases}
	\end{equation}
	
	By the hypotheses and (\ref{eq:Psievol}) from the proof of Proposition \ref{prop:torsionbnd}, we know that if the background metric is K\"ahler, then
	$$\HH |\del_+\del_-(\ga^{-1}(u-u_\infty))|^2_{\ga^{-1}\tilde{\mu}}\leq 2|\del_+\del_-(\ga^{-1}(u-u_\infty))|^2_{\ga^{-1}\tilde{\mu}}.$$
	By the maximum principle, if $|\del_+\del_-(\ga^{-1}(u-u_\infty))|^2_{\ga^{-1}\tilde{\mu}}|_{t=0}\equiv 0$, then $|\del_+\del_-(\ga^{-1}(u-u_\infty))|^2_{\ga^{-1}\tilde{\mu}}\equiv 0$ for all $t$.
	
	By Propositions \ref{prop:lambdalowerbnd} and \ref{prop:Phievol}, we get that for all $\gb \in(0,1)$,
	\begin{equation}\label{eq:KahlerC2}
		0<\min_{M}\left(1+\frac{\ga^{-1}(u_0-u_\infty)_{z\bz}}{e^{F_+/\gb}g}\right) \leq \left(1+\frac{\ga^{-1}(u-u_\infty)_{z\bz}}{e^{F_+/\gb}g}\right) \leq \max_M \left(1+\frac{\ga^{-1}(u_0-u_\infty)_{z\bz}}{e^{F_+/\gb}g}\right).
	\end{equation}
	Thus, $\mu + \Box (u-u_\infty)$ is bounded by $\mu$ uniformly in time and the Li-Yau Harnack inequality implies exponentially fast smooth convergence of $u-u_\infty$ to a constant, from which the result follows.
\end{proof}

\section{Proofs of Theorem \ref{thm:main2} and Corollary \ref{cor:main2}}\label{sec:conv}
\begin{proof}[Proof of Theorem \ref{thm:main2}]
	Since the initial data converges smoothly, the initial data is bounded independent of $\gb$. So, by Proposition \ref{prop:osc}, the flows are all uniformly bounded in $C^0$-norm independent of $i$. Additionally, by definition of the constant $C_0$ (see (\ref{eq:C_0})) and the proof of Proposition \ref{prop:uprbnd},  if $C_0(\gb_i)$ is bounded as $\gb_i\nearrow 1$, then the flow remains uniformly parabolic on compact time intervals and the parabolicity constants are bounded independently of $i$. Thus, Proposition \ref{prop:Holderest} holds with constants which are bounded independent of $i$. So, by bootstrapping, Arzela-Ascoli, and taking a diagonal subsequence, there is a subsequence $(u_{i_j})$ converging in $C^\infty_{loc}(M\times[0,\infty))$ to a limit $u$.
	
	As the limit superior is the supremal subsequential limit,
	$$C_0(u) = \lim_{j\to \infty}C_0(u_{i_j})\leq \limsup_{i\to\infty}C_0(u_i)<\infty,$$
	we know that $u\in \mathcal{A}(\omega_0)$ and $u$ will solve the parabolic twisted Monge-Amp\`ere equation.
\end{proof}

\begin{proof}[Proof of Corollary \ref{cor:main2}]
	By Proposition \ref{prop:lambdalowerbnd} and the hypotheses, 
	$$\min_{M\times [0,\tau)}(1+\frac{\i\del_+\delb_+u_i}{\omega_{0,i}^+}) \geq \frac{1}{K}>0.$$
	So that $C_0(u_i)\leq C(K)$ independent of $i$. Applying Corollary \ref{cor:main} for all $i$ so that $\gb_i>\gb_0$ then immediately yields the result.
\end{proof}

\backmatter

\bmhead{Acknowledgements}

The author would like to acknowledge partial support from NSF-RTG grant DMS-2038103 and thank Hao Fang for many helpful conversations. The author would like to thank Jeffrey Streets and S\'ebastien Picard for feedback on an early version of this paper.

\begin{appendices}

\section{Evolution Equations on $M$}\label{sec:evolonM}
In this appendix, we have collected many of the calculations that are integral to the above results, but whose proofs are relatively tedious. We will attempt to divide this appendix into subsections reflective of where the relevant evolutions are used.

We will use Lemma \ref{lem:Hu} in the proof of \ref{prop:lambdalowerbnd}. 
\begin{lemma}\label{lem:Hu}
	Under the same conditions as before,
	\begin{align}
		L u =&\ \frac{1}{\eta} - \frac{\beta}{\lambda} + (\beta-1)\\
		\HH u =&\ \frac{\del u}{\del t} + \frac{\beta}{\lambda}-\frac{1}{\eta} + (1-\beta)
	\end{align}
	\begin{proof}
		$$L u = \frac{\beta(\lambda-1)}{\lambda} + \frac{(1-\eta)}{\eta} = \frac{1}{\eta} - \frac{\beta}{\lambda} + (\beta - 1).$$
	\end{proof}
\end{lemma}

For the proofs of Propositions \ref{prop:nondeg} and \ref{prop:osc}.
\begin{lemma}\label{lem:dotucaloric}
	Under the same conditions as before, we have
	$$\mathcal{H} \frac{\del u}{\del t} = 0.$$
	\begin{proof}
		We simply differentiate (\ref{eq:pstma}), and obtain
		$$\frac{\del}{\del t} (\frac{\del u}{\del t}) = \beta \frac{1}{g\lambda} (\frac{\del u}{\del t})_{z\bar{z}} + \frac{1}{h\eta}(\frac{\del u}{\del t})_{w\bar{w}} = Lu.$$
	\end{proof}
\end{lemma}

A useful lemma before we proceed in the proof of Lemma \ref{lem:heatlambda} is the following.
\begin{lemma}\label{lem:pluriclosed}
	If $\i \del\delb \omega=0$, then for any $u\in C^4(M)$, $\i \del\delb \omega_u =0$ as well. In local coordinates, this implies 
	\begin{equation}\label{eq:pluriclosed}
		(g\gl)_{w\bw} + (h\eta)_{z\bz}=0.
	\end{equation}
	\begin{proof}
		This is simply a consequence of $\i\del\delb \Box u=0$. To see this, notice that 
		$$\Box u = u_{z\bz}\i dz\wedge d\bz - u_{w\bw}\i dw\wedge d\bw$$
		so that 
		$$\i\del\delb\Box u = (u_{z\bz w\bw} - u_{w\bw z\bz})dz\wedge dw\wedge  d\bz\wedge d\bw =0.$$
	\end{proof}
\end{lemma}

For Propositions \ref{prop:lambdalowerbnd} and \ref{prop:uprbnd} we will require the following lemma.
\begin{lemma}\label{lem:heatlambda}
	Under the same conditions as before, and choosing local holomorphic coordinates $(z,w)$ which are compatible with the splitting, we compute
	\begin{align}
		\HH\lambda =&\ - \frac{\beta}{g} \left|\frac{\lambda_z}{\lambda}\right|^2 + \frac{2}{h\eta}\Re(\frac{g_w\lambda_{\bw}}{g}) + \frac{1}{g} \left|\frac{\eta_z}{\eta}\right|^2 + \frac{2}{g}\Re(\frac{h_z\eta_{\bar{z}}}{h\eta})  + \frac{g_{w\bar{w}}}{gh}\frac{\lambda}{\eta} + \frac{h_{z\bar{z}}}{gh},\\
		\HH \eta =&\ \frac{\beta}{h} \left|\frac{\lambda_w}{\lambda}\right|^2 + \frac{2\beta}{h} \Re(\frac{g_w\lambda_{\bw}}{g\lambda}) + \frac{2\beta}{g\lambda}\Re(\frac{h_z\eta_{\bz}}{h}) - \frac{1}{h}\left|\frac{\eta_w}{\eta}\right|^2 + \beta \frac{h_{z\bz}}{gh}\frac{\eta}{\lambda} + \beta \frac{g_{w\bw}}{gh}.
	\end{align}
	\begin{proof}
		We begin by taking the first derivative in local coordinates.
		$$\frac{\del u_p}{\del t} = \beta \frac{\lambda_p}{\lambda} -  \frac{\eta_p}{\eta}$$
		Taking the second yields
		$$\frac{\del u_{p\bar{p}}}{\del t} = \beta \frac{\lambda_{p\bar{p}}}{\lambda} - \frac{\eta_{p\bar{p}}}{\eta} - \beta \left|\frac{\lambda_p}{\lambda}\right|^2 + \left|\frac{\eta_p}{\eta}\right|^2.$$
		As the background metrics are non-evolving, we can take $p=z$ and convert to $\lambda$.
		$$\frac{\del}{\del t}\lambda = \beta \frac{\lambda_{z\bar{z}}}{g\lambda} - \frac{\eta_{z\bar{z}}}{g\eta} - \frac{\beta}{g} \left|\frac{\lambda_z}{\lambda}\right|^2 + \frac{1}{g} \left|\frac{\eta_z}{\eta}\right|^2.$$
		We can then apply Lemma \ref{lem:pluriclosed} 
		$$\frac{\del}{\del t}\lambda = \beta \frac{\lambda_{z\bar{z}}}{g\lambda} +  \frac{1}{\eta}[\frac{\lambda_{w\bar{w}}}{h} + 2\Re(\frac{g_w}{gh}\lambda_{\bar{w}}) + \frac{g_{w\bar{w}}}{gh}\lambda + 2\Re(\frac{h_z}{gh}\eta_{\bar{z}}) + \eta \frac{h_{z\bar{z}}}{gh}] - \frac{\beta}{g} \left|\frac{\lambda_z}{\lambda}\right|^2 + \frac{1}{g} \left|\frac{\eta_z}{\eta}\right|^2.$$
		
	\end{proof}
\end{lemma}

As a consequence of Lemma \ref{lem:heatlambda} we get several more useful evolution equations. They are stated in the form of the following corollary.
\begin{corollary}\label{cor:lambdacompositions}
	Under the same conditions as in Lemma \ref{lem:heatlambda}, we compute
	\begin{align}
		\mathcal{H}\log \lambda = &\ \frac{1}{h\eta}\left|\frac{\lambda_w}{\lambda} + \frac{g_w}{g}\right|^2 + \frac{1}{g\lambda} \left|\frac{\eta_z}{\eta} + \frac{h_z}{h}\right|^2 + \frac{1}{g\lambda}\left(\frac{h_{z\bar{z}}}{h}  - \left|\frac{h_z}{h}\right|^2\right)  + \frac{1}{h\eta}\left(\frac{g_{w\bar{w}}}{g}  - \left|\frac{g_w}{g}\right|^2\right),  \label{eqn:loglambdaevol}\\
		\HH \log \eta =&\ \beta \HH \log \lambda, \label{eqn:logetaevol}\\
		\HH\frac{1}{\lambda} =&\ - \frac{\beta}{g\lambda^2} \left|\frac{\lambda_z}{\lambda}\right|^2 - \frac{2}{h\lambda \eta} \left|\frac{\lambda_w}{\lambda} + \frac{g_w}{g}\right|^2 - \frac{1}{g\lambda^2} \left|\frac{\eta_z}{\eta} + \frac{h_z}{h}\right|^2\label{eq:inverselambda}\\
		&\  - \frac{1}{g\lambda^2}\left(\frac{h_{z\bar{z}}}{h} -\left|\frac{h_z}{h}\right|^2\right) - \frac{1}{h\lambda\eta}\left(\frac{g_{w\bar{w}}}{g} - \left|\frac{g_w}{g}\right|^2  \right),  \nonumber\\
		\HH\frac{1}{\eta} =&\ -\frac{\beta}{h\eta^2}\left|\frac{\lambda_w}{\lambda} + \frac{g_w}{g}\right|^2 - \frac{2\beta}{g\lambda\eta}\left|\frac{\eta_z}{\eta} + \frac{h_z}{h}\right|^2 - \frac{1}{h\eta^2}\left|\frac{\eta_w}{\eta}\right|^2\label{eq:inverseeta}\\
		&\ - \frac{\beta}{g\lambda\eta}\left(\frac{h_{z\bz}}{h} -\left|\frac{h_z}{h}\right|^2\right) -  \frac{\beta}{h\eta^2}\left(\frac{g_{w\bw}}{g} - \left|\frac{g_w}{g}\right|^2\right). \nonumber
	\end{align}
	\begin{proof}
		Equations (\ref{eqn:loglambdaevol}) \& (\ref{eq:inverselambda}) follow from Lemma \ref{lem:heatlambda} and
		$$\HH\log \lambda = \frac{\HH\lambda}{\lambda} + \beta\frac{|\lambda_z|^2}{g\lambda^3} + \frac{|\lambda_w|^2}{h\eta\lambda^2},$$
		
		$$\HH\frac{1}{\lambda} = -\frac{1}{\lambda^2}\HH\lambda - 2\beta \frac{|\lambda_z|^2}{g\lambda^4} - 2 \frac{|\lambda_w|^2}{h\eta\lambda^3}.$$
		
		As a consequence of Lemma \ref{lem:dotucaloric} and (\ref{eq:pstma}), we have (\ref{eqn:logetaevol}). Equation (\ref{eq:inverseeta}) follows from Lemma \ref{lem:heatlambda} in much the same way as (\ref{eq:inverselambda}).
	\end{proof}
\end{corollary}

\section{Mixed Derivative Evolution Equations}\label{sec:torsionpot}
To estimate the norm of the form $\del_+\del_- u$, it will be convenient to use a Bochner formula argument. As such, we will need to introduce an adjusted metric $\tilde{\omega}$ whose Chern Laplacian on functions matches up with the linearized operator of (\ref{eq:pstma}). Namely, we define 
\begin{equation}\label{eq:tildemetric}
	\tilde{\omega} := \frac{\lambda}{\beta}\omega_0^+ +\eta \omega_0^-.
\end{equation} 
The Chern Laplacian can be extended to sections of tensor bundles in the following way.
\begin{definition}\label{defn:cxnChernLap}
	The \textit{rough Chern Laplacian} of $\tilde{\omega}$ is defined in local holomorphic coordinates by 
	$$\tilde{\Delta} X = \tilde{\omega}^{\bj i}\tilde{\nabla}_i\bar{\tilde{\nabla}}_j X,$$
	where $X$ is a section of a tensor bundle and $\tilde{\nabla}$ is the Chern connection of $\tilde{\omega}$.
\end{definition}

We begin by computing the evolution of the mixed derivative as a form. 
\begin{lemma}\label{lem:Hnu}
	Supposing as before, $\del_+\del_- u$ satisfies 
	$$(\dt-\tilde{\Delta})\del_+\del_-u = \Psi\in \Lambda^{2,0}$$
	where $\Psi = \Psi_{zw}dz\wedge dw$ is given in local holomorphic coordinates adapted to the splitting by 
	\begin{align}
		\Psi_{zw} =&\ \left(h_{zw} -h^{-1}h_zh_w - g^{-1}h_zg_w\right)\frac{(\eta-1)}{h\eta} \label{eq:defPsi}\\
		&\ + \beta\left(-g_{zw} + g^{-1}g_zg_w + h^{-1}h_zg_w\right)\frac{(\lambda-1)}{g\lambda}\nonumber \\
		&\ + \left((\beta-1) \frac{\eta_zg_w}{g\eta} - \beta \frac{\eta_zg_w}{g\lambda\eta} + \frac{h_z\eta_w}{h\eta^2}\right)\nonumber\\
		&\ + \left((\beta -1) \frac{h_z\lambda_w}{h\lambda} -\beta \frac{\lambda_zg_w}{g\lambda^2} + \frac{h_z\lambda_w}{h\lambda\eta}\right) + (\beta-1)\frac{\eta_z\lambda_w}{\eta\lambda}.\nonumber
	\end{align}
	\begin{proof}
		The non-vanishing connection coefficients for the Chern connection of $\tilde{\omega}$ can be computed to be
		\begin{equation}\label{eq:concoefs}
			\Gamma_{iz}^z = \frac{g_i}{g} + \frac{\lambda_i}{\lambda},\quad \Gamma_{iw}^w = \frac{h_i}{h} + \frac{\eta_i}{\eta}.
		\end{equation}
		Therefore, the rough Chern Laplacian acts on the mixed derivative by
		\begin{align}
			\tilde{\Delta} \del_+\del_- u =&\ \tilde{\omega}^{\bj i}\tilde{\nabla}_i\bar{\tilde{\nabla}}_j (u_{zw}dz\wedge dw)\\
			=&\ \tilde{\omega}^{\bj i}(u_{zwi\bj} -(\Gamma_{iz}^z+\Gamma_{iw}^w)u_{zw\bj})dz\wedge dw\nonumber\\
			=&\ (Lu_{zw} - \tilde{\omega}^{\bj i}(\Gamma_{iz}^z+\Gamma_{iw}^w)u_{zw\bj})dz\wedge dw.\label{eq:roughlap}
		\end{align}
		
		Including the time evolution and expanding out the traces, we obtain the rough heat flow of $\del_+\del_- u$ in terms of the scalar heat flow of $u_{zw}$ in local coordinates.
		\begin{align}
			(\dt -\tilde{\Delta}) \del_+\del_- u =&\ (\HH u_{zw} + \frac{\beta}{g\lambda}(\Gamma_{zz}^z+\Gamma_{zw}^w)u_{zw\bz} + \frac{1}{h\eta}(\Gamma_{wz}^z+\Gamma_{ww}^w)u_{zw\bw})dz\wedge dw\nonumber\\
			=&\ \HH u_{zw}dz\wedge dw\label{eq:Hupm}\\
			&\  + \frac{\beta}{g\lambda}(\frac{g_z}{g} + \frac{\lambda_z}{\lambda} + \frac{h_z}{h} + \frac{\eta_z}{\eta})(g(\lambda-1))_wdz\wedge dw\nonumber\\
			&\ + \frac{1}{h\eta}(\frac{g_w}{g} + \frac{\lambda_w}{\lambda}+ \frac{h_w}{h} + \frac{\eta_w}{\eta})(h(1-\eta))_zdz\wedge dw\nonumber
		\end{align}
		
		To compute $\HH u_{zw}$, we will need the following formula which is computed by differentiating (\ref{eq:pstma}) twice similarly to the proof of Lemma \ref{lem:heatlambda}.
		\begin{equation}\label{eq:dotuzw}
			\frac{\del u}{\del t}_{zw} = \beta\frac{\lambda_{zw}}{\lambda} - \frac{\eta_{zw}}{\eta} - \beta\frac{\lambda_z\lambda_w}{\lambda^2}+\frac{\eta_z\eta_w}{\eta^2}
		\end{equation}
		Commuting coordinate derivatives yields 
		\begin{equation}\label{eq:Luzw}
			Lu_{zw} = \frac{\beta}{g\lambda}(g(\lambda-1))_{zw} + \frac{1}{h\eta}(h(1-\eta))_{zw}.
		\end{equation}
		Expanding (\ref{eq:Luzw}) and subtracting it from (\ref{eq:dotuzw}), we obtain
		$$\HH u_{zw} = \left(\frac{\eta_z\eta_w}{\eta^2} + \frac{h_z\eta_w + h_w\eta_z + h_{zw}(\eta-1)}{h\eta}\right) -\beta\left(\frac{\lambda_z\lambda_w}{\lambda^2} + \frac{g_z\lambda_w + g_w\lambda_z + g_{zw}(\lambda-1)}{g\lambda}\right).$$
		Plugging this back into (\ref{eq:Hupm}) and simplifying yields
		\begin{align*}
			[(\dt-\tilde{\Delta}) \del_+\del_-u]_{zw} =&\ \left(h_{zw} -h^{-1}h_zh_w - g^{-1}h_zg_w\right)\frac{(\eta-1)}{h\eta}\\
			&\ + \beta\left(-g_{zw} + g^{-1}g_zg_w + h^{-1}h_zg_w\right)\frac{(\lambda-1)}{g\lambda} \\
			&\ + \left((\beta-1) \frac{\eta_zg_w}{g\eta} - \beta \frac{\eta_zg_w}{g\lambda\eta} + \frac{h_z\eta_w}{h\eta^2}\right)\\
			&\ + \left((\beta -1) \frac{h_z\lambda_w}{h\lambda} -\beta \frac{\lambda_zg_w}{g\lambda^2} + \frac{h_z\lambda_w}{h\lambda\eta}\right) + (\beta-1)\frac{\eta_z\lambda_w}{\eta\lambda}.
		\end{align*}			
		Comparing the right-hand side to $\Psi$, it is clear that this finishes the calculation.
	\end{proof}
\end{lemma}

To make use of Lemma \ref{lem:Hnu}, we must compute how a mixed-type $(2,0)$-form evolving by a heat equation behaves. As $\Lambda_+^{1,0}\wedge \Lambda^{1,0}_-$ is a line bundle over $M^2$, we can make use of a Bochner formula.
\begin{lemma}\label{lem:Hnu2}
	Let $\nu \in \Lambda_+^{1,0}\otimes \Lambda^{1,0}_-$ and $\tilde{\omega}$ be the adjusted metric associated to $\omega$. If the heat operator has been calculated as 
	$$(\dt - \tilde{\Delta})\nu = \Psi,$$
	then the norm satisfies 
	\begin{equation}\label{eq:torsionpotnorm}
		\HH|\nu|^2_{\tilde{\omega}} = - |\delb \nu|^2_{\tilde{\omega}} - |\tilde{\nabla} \nu|^2_{\tilde{\omega}} - |\nu|^2_{\tilde{\omega}} \HH\log ((g\lambda)(h\eta)) + 2\Re(\langle \Psi,\bnu\rangle_{\tilde{\omega}}).
	\end{equation}
	\begin{proof}
		By the choice of base metric, the rough Chern Laplacian $\tilde{\Delta}$, when restricted to functions, will yield the linearized operator $L$. Thus, by a Bochner formula, we have
		$$L|\nu|^2_{\tilde{\omega}} = 2\Re (\tilde{\omega}(\tilde{\Delta} \nu,\bnu)) + |\delb \nu|^2_{\tilde{\omega}} + |\tilde{\nabla} \nu|^2_{\tilde{\omega}} -  (\Lambda_{\tilde{\omega}}F_{\tilde{\omega}})|\nu|^2_{\tilde{\omega}}$$
		where  $F_{\tilde{\omega}}$ is the curvature of the induced metric on $\Lambda_+^{1,0}\otimes \Lambda^{1,0}_-$. The proof is a calculation along the lines of other Bochner formulas and is omitted.
		
		In a coordinate neighborhood, the line bundle $\Lambda_+^{1,0}\otimes \Lambda^{1,0}_-$ admits a holomorphic section $dz\wedge dw$, so we can compute 
		$$\hat{F} = -\Lambda_{\tilde{\omega}} \del\delb\log |dz\wedge dw|^2_{\tilde{\omega}} = - L \log |dz\wedge dw|^2_{\tilde{\omega}}.$$
		
		The norm evolution is computed to be 
		$$\dt|\nu|^2_{\tilde{\omega}} = |\nu|_{\tilde{\omega}}^2\dt\log |dz\wedge dw|^2_{\tilde{\omega}} + 2\Re(\tilde{\omega}(\dt \nu,\bnu)).$$
		Subtracting these gives the result.			
	\end{proof}
\end{lemma}

We will specialize Lemma \ref{lem:Hnu2} to obtain an estimate for $\HH|\del_+\del_-u|^2_{\tilde{\omega}}$ by estimating the more difficult terms now. In what follows, we will denote the background torsion $(2,1)$-form by $T^0$, in coordinates 
$$T^0 := -\i \del\omega_0 =  -g_w dz\wedge dw\wedge d\bz + h_z dz\wedge dw\wedge d\bw,\quad |T^0|^2_{\omega_0} = \frac{1}{h}\left|\frac{g_w}{g}\right|^2 + \frac{1}{g}\left|\frac{h_z}{h}\right|^2.$$ 

\begin{lemma}\label{lem:Psinu}
	Letting $\Psi \in \Lambda^{2,0}$ be defined as in (\ref{eq:defPsi})
	\begin{align}
		-2\Re\langle\Psi,\bar{\del_+\del_-u} \rangle \leq &\  C_3|\del_+\del_-u|_{\tilde{\omega}} +2\sqrt{\beta}(1-\beta)\sqrt{\frac{1}{g\lambda}}\left|\frac{\eta_z}{\eta}+\frac{h_z}{h}\right|\sqrt{\frac{1}{h\eta}}\left|\frac{\lambda_w}{\lambda}+\frac{g_w}{g}\right||\del_+\del_-u|_{\tilde{\omega}} \label{eq:Psiestfinal}\\
		&\ +2\left(\sqrt{\frac{1}{h\eta}}\left|\frac{\beta g_w}{g\lambda}\right|\right)\left(\sqrt{\frac{\beta}{g\lambda}}\left|\frac{\eta_z}{\eta}+\frac{h_z}{h}\right||\del_+\del_-u|_{\tilde{\omega}}\right)\nonumber\\
		&\ + 2 \left(\sqrt{\frac{\beta}{g\lambda}}\left|\frac{h_z}{h\eta}\right|\right)\left(\sqrt{\frac{1}{h\eta}}\left|\frac{\lambda_w}{\lambda} + \frac{g_w}{g}\right||\del_+\del_-u|_{\tilde{\omega}}\right)\nonumber\\
		&\ +2\left(\sqrt{\frac{\beta}{gh\lambda\eta}}\left|\frac{h_z\eta_w}{h\eta^2}\right|\right)|\del_+\del_-u|_{\tilde{\omega}} + 2\left(\sqrt{\frac{\beta}{gh\lambda\eta}}\left|\beta\frac{g_w\lambda_z}{g\lambda^2}\right|\right)|\del_+\del_-u|_{\tilde{\omega}}\nonumber
	\end{align}
	where $C_3$ is a positive constant defined over the course of the proof which depends only on background data and vanishes when $\omega_0$ is K\"ahler.
	\begin{proof}
		We begin by expanding terms using the definition of $\Psi$ in (\ref{eq:defPsi}) and estimating the first two lines by $C_1|\del_+\del_-u|_{\tilde{\omega}}$. This follows from the Cauchy-Schwarz inequality with $C_1 = C_0\max_M |\nabla^{\omega_0} T^0|_{\omega_0}$ where $C_0$ is defined at (\ref{eq:C_0}).
		\begin{align}
			-2\Re\langle\Psi,\bar{\del_+\del_-u} \rangle \leq &\ C_1|\del_+\del_-u|_{\tilde{\omega}} -2\frac{\beta}{gh\lambda\eta} \Re\left( \left((\beta-1) \frac{\eta_zg_w}{g\eta} - \beta \frac{\eta_zg_w}{g\lambda\eta} + \frac{h_z\eta_w}{h\eta^2}\right)u_{\bz\bw}\right)\label{eq:Psiest2}\\
			&\ -2\frac{\beta}{gh\lambda\eta} \Re\left( \left((\beta -1) \frac{h_z\lambda_w}{h\lambda} -\beta \frac{\lambda_zg_w}{g\lambda^2} + \frac{h_z\lambda_w}{h\lambda\eta}\right)u_{\bz\bw}\right)\nonumber\\
			&\ -2\frac{\beta}{gh\lambda\eta} \Re\left( (\beta-1)\frac{\eta_z\lambda_w}{\eta\lambda}u_{\bz\bw}\right)\nonumber
		\end{align}		
		We hope to balance some of these terms against $|\del_+\del_-u|^2_H \HH\log |dz\wedge dw|^2_H$ from Lemma \ref{lem:Hnu2}, so we group relevant terms to make any potential cancellations more obvious.
		\begin{align}
			-2\Re\langle\Psi,\bar{\del_+\del_-u} \rangle \leq &\ C_2|\del_+\del_-u|_{\tilde{\omega}} -2\frac{\beta}{gh\lambda\eta} \Re\left(\left( \frac{h_z\eta_w}{h\eta^2} -\beta \frac{\lambda_zg_w}{g\lambda^2} \right)u_{\bz\bw}\right)\label{eq:Psiest4}\\
			&\ +2\frac{\beta(1-\gb)}{gh\lambda\eta} \Re\left(\left(\frac{\eta_z}{\eta} + \frac{h_z}{h}\right)\left(\frac{\lambda_w}{\lambda}+\frac{g_w}{g}\right)u_{\bz\bw}\right)\nonumber \\
			&\ +2\frac{\beta}{gh\lambda\eta} \Re\left( \frac{\beta}{\lambda}\frac{\eta_zg_w}{\eta g}u_{\bz\bw}\right) -2\frac{\beta}{gh\lambda\eta} \Re\left(\frac{1}{\eta}\frac{h_z\lambda_w}{h\lambda}u_{\bz\bw}\right),\nonumber
		\end{align}
		where the constant $C_2 = C_1 + \beta(1-\beta)C_0\max_M|T^0|^2_{\omega_0}$.
		
		Applying the Cauchy-Schwarz and Young inequalities and rearranging (\ref{eq:Psiest4}) a bit gives 
		\begin{align}
			-2\Re\langle\Psi,\bar{\del_+\del_-u} \rangle \leq &\ C_3|\del_+\del_-u|_{\tilde{\omega}} +2\sqrt{\beta}(1-\beta)\sqrt{\frac{1}{g\lambda}}\left|\frac{\eta_z}{\eta}+\frac{h_z}{h}\right|\sqrt{\frac{1}{h\eta}}\left|\frac{\lambda_w}{\lambda}+\frac{g_w}{g}\right||\del_+\del_-u|_{\tilde{\omega}} \label{eq:Psiest6}\\
			&\ +2\left(\sqrt{\frac{1}{h\eta}}\left|\frac{\beta g_w}{g\lambda}\right|\right)\left(\sqrt{\frac{\beta}{g\lambda}}\left|\frac{\eta_z}{\eta}+\frac{h_z}{h}\right||\del_+\del_-u|_{\tilde{\omega}}\right)\nonumber\\
			&\ + 2 \left(\sqrt{\frac{\beta}{g\lambda}}\left|\frac{h_z}{h\eta}\right|\right)\left(\sqrt{\frac{1}{h\eta}}\left|\frac{\lambda_w}{\lambda} + \frac{g_w}{g}\right||\del_+\del_-u|_{\tilde{\omega}}\right)\nonumber\\
			&\ -2\frac{\beta}{gh\lambda\eta} \Re\left(\left( \frac{h_z\eta_w}{h\eta^2} -\beta \frac{\lambda_zg_w}{g\lambda^2} \right)u_{\bz\bw}\right).\nonumber
		\end{align}
		where $C_3 = C_2 + \beta C_0^3\max_M|T_0|^2_{\omega_0}$.
		
		A final application of the Cauchy-Schwarz inequality and regrouping proves the lemma.
	\end{proof}
\end{lemma}

We can then combine the results of Lemmas \ref{lem:Hnu} and \ref{lem:Psinu} to obtain the following lemma.
\begin{lemma}\label{lem:Hnunormsquared} For every $\beta \in(0,1)$, $\ge>0$, and $\gd\in(0,1)$,
	\begin{align}
		\HH |\del_+\del_-u|_{\tilde{\omega}}^2 \leq &\ C_8\left(\frac{1}{\ge}+\frac{1}{\delta}-\beta^2\right) + C_3|\del_+\del_-u|_{\tilde{\omega}} + C_6|\del_+\del_-u|^2_{\tilde{\omega}} +C_7\left(\frac{1}{h\eta^2}\left|\frac{\eta_w}{\eta}\right|^2 + \frac{1}{g\lambda^2}\left|\frac{\lambda_z}{\lambda}\right|^2\right)\label{eq:Hnunormsquared}\\
		&\ + [-\beta^2(1-\delta) + \sqrt{\beta}(1-\beta)|\del_+\del_-u|_{\tilde{\omega}} - (1+\beta-\ge)|\del_+\del_-u|^2_{\tilde{\omega}}]\\
		&\ \times \left(\frac{1}{h\eta}\left|\frac{\lambda_w}{\lambda} + \frac{g_w}{g}\right|^2 + \frac{1}{g\lambda} \left|\frac{\eta_z}{\eta} + \frac{h_z}{h}\right|^2\right)\nonumber.
	\end{align}
	The constant $C_3$ is the same as in Lemma \ref{lem:Psinu} and $C_6,C_7,C_8\geq 0$ are defined over the course of the proof and depend only on $\gb$ and $\omega_0$. Moreover, when $\omega_0$ is K\"ahler $C_6=2$ and $C_7=C_8=0$.
	\begin{proof}
		By (\ref{eq:torsionpotnorm}), we have
		$$\HH |\del_+\del_-u|_{\tilde{\omega}}^2 = -|\delb\del_+\del_-u|^2_{\tilde{\omega}} - |\nabla \del_+\del_-u|_{\tilde{\omega}}^2 - ((1+\beta)\HH\log \lambda - L\log gh)|\del_+\del_-u|^2_{\tilde{\omega}} - 2\Re\langle\Psi,\bar{\del_+\del_-u}\rangle$$
		with $\Psi\in\Lambda^{2,0}$ given as in (\ref{eq:defPsi}).
		
		Corollary \ref{cor:lambdacompositions} implies
		$$
		(1+\beta)\HH\log\lambda - L\log gh \geq \frac{1+\beta}{h\eta}\left|\frac{\lambda_w}{\lambda} + \frac{g_w}{g}\right|^2 + \frac{1+\beta}{g\lambda} \left|\frac{\eta_z}{\eta} + \frac{h_z}{h}\right|^2\\
		-C_4\left(\frac{1}{\lambda} + \frac{1}{\eta}\right),
		$$
		where $C_4 = \max\{|\Omega_{w\bw z}^z(\omega_0) - \beta \Omega_{z\bz z}^z(\omega_0)|,|\beta\Omega_{z\bz w}^w(\omega_0) - \Omega_{w\bw w}^w(\omega_0)|\}$, and a standard calculation yields
		\begin{equation}\label{eq:delbnu}
			|\delb \del_+\del_-u|_{\tilde{\omega}}^2
			= \frac{\beta^2}{h\eta}\left|\frac{\lambda_w}{\lambda} + \frac{g_w}{g} - \frac{g_w}{g\lambda}\right|^2 + \frac{\beta}{g\lambda}\left|\frac{\eta_z}{\eta} + \frac{h_z}{h} - \frac{h_z}{h\eta}\right|^2.
		\end{equation}
		
		Proposition \ref{prop:lambdalowerbnd} allows us to define $C_5 = C_0C_4$ ($C_0$ is defined at (\ref{eq:C_0})) so that 
		\begin{align*}
			\HH |\del_+\del_-u|_{\tilde{\omega}}^2 \leq &\ C_5|\del_+\del_-u|^2_{\tilde{\omega}} - \frac{\beta^2}{h\eta}\left|\frac{\lambda_w}{\lambda} + \frac{g_w}{g} - \frac{g_w}{g\lambda}\right|^2 - \frac{\beta}{g\lambda}\left|\frac{\eta_z}{\eta} + \frac{h_z}{h} - \frac{h_z}{h\eta}\right|^2\\
			&\ - (1+\beta)\left(\frac{1}{h\eta}\left|\frac{\lambda_w}{\lambda} + \frac{g_w}{g}\right|^2 + \frac{1}{g\lambda} \left|\frac{\eta_z}{\eta} + \frac{h_z}{h}\right|^2\right)|\del_+\del_-u|^2_{\tilde{\omega}}\\
			&\ - 2\Re\langle\Psi,\bar{\del_+\del_-u}\rangle_{\tilde{\omega}}.
		\end{align*}
		Plugging in the previous estimates, we find
		\begin{align*}
			\HH |\del_+\del_-u|_{\tilde{\omega}}^2 \leq &\ C_3|\del_+\del_-u|_{\tilde{\omega}} + C_5|\del_+\del_-u|^2_{\tilde{\omega}} - \frac{\beta^2}{h\eta}\left|\frac{\lambda_w}{\lambda} + \frac{g_w}{g} - \frac{g_w}{g\lambda}\right|^2 - \frac{\beta}{g\lambda}\left|\frac{\eta_z}{\eta} + \frac{h_z}{h} - \frac{h_z}{h\eta}\right|^2\\
			&\ - (1+\beta)\left(\frac{1}{h\eta}\left|\frac{\lambda_w}{\lambda} + \frac{g_w}{g}\right|^2 + \frac{1}{g\lambda} \left|\frac{\eta_z}{\eta} + \frac{h_z}{h}\right|^2\right)|\del_+\del_-u|^2_{\tilde{\omega}}\\
			&\ +2\sqrt{\beta}(1-\beta)\sqrt{\frac{1}{g\lambda}}\left|\frac{\eta_z}{\eta}+\frac{h_z}{h}\right|\sqrt{\frac{1}{h\eta}}\left|\frac{\lambda_w}{\lambda}+\frac{g_w}{g}\right||\del_+\del_-u|_{\tilde{\omega}}\\
			&\ +2\left(\sqrt{\frac{1}{h\eta}}\left|\frac{\beta g_w}{g\lambda}\right|\right)\left(\sqrt{\frac{\beta}{g\lambda}}\left|\frac{\eta_z}{\eta}+\frac{h_z}{h}\right||\del_+\del_-u|_{\tilde{\omega}}\right)\\
			&\ + 2 \left(\sqrt{\frac{\beta}{g\lambda}}\left|\frac{h_z}{h\eta}\right|\right)\left(\sqrt{\frac{1}{h\eta}}\left|\frac{\lambda_w}{\lambda} + \frac{g_w}{g}\right||\del_+\del_-u|_{\tilde{\omega}}\right)\\
			&\ +2\left(\sqrt{\frac{\beta}{gh\lambda\eta}}\left|\frac{h_z\eta_w}{h\eta^2}\right|\right)|\del_+\del_-u|_{\tilde{\omega}} + 2\left(\sqrt{\frac{\beta}{gh\lambda\eta}}\left|\beta\frac{g_w\lambda_z}{g\lambda^2}\right|\right)|\del_+\del_-u|_{\tilde{\omega}}.
		\end{align*}
		
		Applying Young's inequality to the third line, applying Young's inequality with coefficient $\ge>0$ to the fourth line, and using the fact that $\beta\in (0,1]$ yields 
		\begin{align}
			\HH |\del_+\del_-u|_{\tilde{\omega}}^2 \leq &\ C_3|\del_+\del_-u|_{\tilde{\omega}} + C_5|\del_+\del_-u|^2_{\tilde{\omega}} - \frac{\beta^2}{h\eta}\left|\frac{\lambda_w}{\lambda} + \frac{g_w}{g} - \frac{g_w}{g\lambda}\right|^2 - \frac{\beta^2}{g\lambda}\left|\frac{\eta_z}{\eta} + \frac{h_z}{h} - \frac{h_z}{h\eta}\right|^2 \label{eq:Hnu1}\\
			&\ + [\sqrt{\beta}(1-\beta)|\del_+\del_-u|_{\tilde{\omega}} - (1+\beta-\ge)|\del_+\del_-u|^2_{\tilde{\omega}}]\left(\frac{1}{h\eta}\left|\frac{\lambda_w}{\lambda} + \frac{g_w}{g}\right|^2 + \frac{1}{g\lambda} \left|\frac{\eta_z}{\eta} + \frac{h_z}{h}\right|^2\right)\nonumber\\
			&\ +2\left(\sqrt{\frac{\beta}{gh\lambda\eta}}\left|\frac{h_z\eta_w}{h\eta^2}\right|\right)|\del_+\del_-u|_{\tilde{\omega}} + 2\left(\sqrt{\frac{\beta}{gh\lambda\eta}}\left|\beta\frac{g_w\lambda_z}{g\lambda^2}\right|\right)|\del_+\del_-u|_{\tilde{\omega}}\nonumber\\
			&\ +\frac{1}{\ge h\eta}\left|\frac{\beta g_w}{g\lambda}\right|^2 + \frac{\beta}{\ge g\lambda}\left|\frac{h_z}{h\eta}\right|^2. \nonumber
		\end{align}
		
		By the Cauchy-Schwarz and Young's inequalities  applied to (\ref{eq:Hnu1}) with $\gd \in(0,1)$, we obtain
		\begin{align*}
			\HH |\del_+\del_-u|_{\tilde{\omega}}^2 \leq &\ C_3|\del_+\del_-u|_{\tilde{\omega}} + C_5|\del_+\del_-u|^2_{\tilde{\omega}} \\
			&\ + [-\beta^2(1-\delta) + \sqrt{\beta}(1-\beta)|\del_+\del_-u|_{\tilde{\omega}} - (1+\beta-\ge)|\del_+\del_-u|^2_{\tilde{\omega}}]\\
			&\ \times \left(\frac{1}{h\eta}\left|\frac{\lambda_w}{\lambda} + \frac{g_w}{g}\right|^2 + \frac{1}{g\lambda} \left|\frac{\eta_z}{\eta} + \frac{h_z}{h}\right|^2\right)\\
			&\ +2\left(\sqrt{\frac{\beta}{gh\lambda\eta}}\left|\frac{h_z\eta_w}{h\eta^2}\right|\right)|\del_+\del_-u|_{\tilde{\omega}} + 2\left(\sqrt{\frac{\beta}{gh\lambda\eta}}\left|\beta\frac{g_w\lambda_z}{g\lambda^2}\right|\right)|\del_+\del_-u|_{\tilde{\omega}}\\
			&\ +\left(\frac{1}{\ge}+\frac{1}{\delta}-\beta^2\right)\left(\frac{1}{h\eta}\left|\frac{g_w}{g\lambda}\right|^2 + \frac{1}{g\lambda}\left|\frac{h_z}{h\eta}\right|^2\right).
		\end{align*}
		
		We can once more apply the standard Young's inequality to the third line to obtain
		\begin{align*}
			\HH |\del_+\del_-u|_{\tilde{\omega}}^2 \leq &\ C_3|\del_+\del_-u|_{\tilde{\omega}} + C_6|\del_+\del_-u|^2_{\tilde{\omega}} +\frac{\beta}{g\lambda\eta}\left|\frac{h_z}{h}\right|^2\left(\frac{1}{h\eta^2}\left|\frac{\eta_w}{\eta}\right|^2\right) + \frac{\beta^3}{h\lambda\eta}\left|\frac{g_w}{g}\right|^2\left(\frac{1}{g\lambda^2}\left|\frac{\lambda_z}{\lambda}\right|^2\right) \\
			&\ + [-\beta^2(1-\delta) + \sqrt{\beta}(1-\beta)|\del_+\del_-u|_{\tilde{\omega}} - (1+\beta-\ge)|\del_+\del_-u|^2_{\tilde{\omega}}]\\
			&\ \times \left(\frac{1}{h\eta}\left|\frac{\lambda_w}{\lambda} + \frac{g_w}{g}\right|^2 + \frac{1}{g\lambda} \left|\frac{\eta_z}{\eta} + \frac{h_z}{h}\right|^2\right)\\
			&\ +\left(\frac{1}{\ge}+\frac{1}{\delta}-\beta^2\right)\left(\frac{1}{h\eta}\left|\frac{g_w}{g\lambda}\right|^2 + \frac{1}{g\lambda}\left|\frac{h_z}{h\eta}\right|^2\right),
		\end{align*}
		where $C_6 = C_5 + 2$.
		
		Setting 
		\begin{align*}
			C_7 =&\ \max\{\max_Mg^{-1},\max_Mh^{-1}\}C_0^2\max_M|T_0|_{\omega_0}^2,\\
			C_8 =&\ C_0^3\max_M\left(\frac{1}{h}\left|\frac{g_w}{g}\right|^2 + \frac{1}{g}\left|\frac{h_z}{h}\right|^2\right),
		\end{align*}
		we obtain the result.
	\end{proof}
\end{lemma}

\section{Evolution Equations for the Local PDE}\label{sec:locpde}
In this lemma, we check that $\frac{\del u}{\del t}$ and $W$ have the (sub)solution properties necessary for the proof of Proposition \ref{prop:Holderest}.
\begin{lemma}\label{lem:ek}
	For any $C^4$ admissible solution $u$ to (\ref{eq:flatpstMA}), we have that 
	$$\HH \frac{\del u}{\del t} = 0.$$ 
	Additionally, for all $v\in \mathbb{C}^2$, the matrix $W$, defined in (\ref{eq:formalpartialleg}), satisfies 
	$$\HH W(v,\bv)\leq 0.$$ 
	\begin{proof}
		For the first part, we compute 
		$$\dt \frac{\del u}{\del t}= \beta \frac{\dt \lambda}{\lambda} - \frac{\dt \eta}{\eta} = \frac{\gb}{\gl}(\frac{\del u}{\del t})_{z\bz} + \frac{1}{\eta}(\frac{\del u}{\del t})_{w\bw} = L\dot{u}.$$
		
		For the second part, the computation is a little trickier. We will begin by computing several identities. First, notice that (\ref{eq:flatpstMA}) implies that for any coordinate directions $i$ and $j$,
		\begin{equation}\label{eq:secondders}
			\HH u_{ij}= -\gb\frac{\gl_i\gl_j}{\gl^2} + \frac{\eta_i\eta_j}{\eta^2}.
		\end{equation}
		Using (\ref{eq:secondders}) we find that $\gl$, $\eta$, and $u_{z\bw}$ satisfy
		\begin{align}
			\HH\gl =&\ -\gb \left|\frac{\gl_z}{\gl}\right|^2 + \left|\frac{\eta_z}{\eta}\right|^2 \label{eq:Hglflat}\\
			\HH u_{z\bw}=&\ -\gb\frac{\gl_z\gl_{\bw}}{\gl^2} + \frac{\eta_z\eta_{\bw}}{\eta^2} \label{eq:Huzwflat}\\
			\HH\eta =&\ \gb \left|\frac{\gl_w}{\gl}\right|^2 - \left|\frac{\eta_w}{\eta}\right|^2 \label{eq:Hetaflat}
		\end{align}
		
		Applying the usual chain and product rules for second order linear parabolic operators to (\ref{eq:Hetaflat}) and (\ref{eq:Huzwflat}), along with the algebraic identity 
		$$|a+b|^2 = |a|^2 + 2\Re(a\bb) + |b|^2,$$ 
		we find
		\begin{align}
			\HH \eta^{-1} =&\ -\frac{\gb}{\eta^2}\left|\frac{\gl_w}{\gl}\right|^2 - \frac{2\gb}{\gl\eta}\left|\frac{\eta_z}{\eta}\right|^2 - \frac{1}{\eta^2}\left|\frac{\eta_w}{\eta}\right|^2 \label{eq:Hetam1}\\
			\HH|u_{z\bw}|^2 =&\ -\frac{2\gb}{\gl^2}\Re(u_{\bz w}\gl_z\gl_{\bw})+ \frac{2}{\eta^2}\Re(u_{\bz w} \eta_z\eta_{\bw})\label{eq:Huzw2}\\
			&\  - \frac{\gb}{\gl}(|\lambda_w|^2 + |u_{zz\bw}|^2)- \frac{1}{\eta}(|u_{ww\bz}|^2 + |\eta_z|^2)\nonumber
		\end{align}
		Finally, we use the product rule to combine (\ref{eq:Hetam1}), and (\ref{eq:Huzw2}) to obtain 
		\begin{align}
			\HH\frac{|u_{z \bw}|^2}{\eta} =&\ \frac{2}{\eta^3}\Re(u_{\bz w} \eta_z\eta_{\bw}) - \frac{\gb}{\gl\eta}(|\lambda_w|^2 + |u_{zz\bw}|^2)- \frac{1}{\eta^2}(|u_{ww\bz}|^2 + |\eta_z|^2)\label{eq:Huzwovereta1}\\
			&\ + |u_{z\bw}|^2 \left(-\frac{\gb}{\eta^2}\left|\frac{\gl_w}{\gl}\right|^2 - \frac{2\gb}{\gl\eta}\left|\frac{\eta_z}{\eta}\right|^2 - \frac{1}{\eta^2}\left|\frac{\eta_w}{\eta}\right|^2\right)\nonumber\\
			&\ -\frac{2\gb}{\gl^2\eta}\Re(u_{\bz w}\gl_z\gl_{\bw}) - \frac{2\gb}{\gl}\Re((|u_{z\bw}|^2)_z(\eta^{-1})_{\bz}) - \frac{2}{\eta}\Re((|u_{z\bw}|^2)_w(\eta^{-1})_{\bw})\nonumber
		\end{align}
		
		We can simplify the terms in the last line of (\ref{eq:Huzwovereta1}) using the identities
		$$(|u_{z\bw}|^2)_z = u_{w\bz}u_{zz\bw} + \gl_w u_{z\bw},\quad (|u_{z\bw}|^2)_w = u_{z\bw}u_{ww\bz} - u_{w\bz}\eta_z.$$
		This yields
		\begin{align}
			\HH\frac{|u_{z \bw}|^2}{\eta} =&\ - \frac{\gb}{\gl\eta}(|\lambda_w|^2 + |u_{zz\bw}|^2)- \frac{1}{\eta^2}(|u_{ww\bz}|^2 + |\eta_z|^2) \label{eq:Huzwovereta2}\\
			&\ + |u_{z\bw}|^2 \left(-\frac{\gb}{\eta^2}\left|\frac{\gl_w}{\gl}\right|^2 - \frac{2\gb}{\gl\eta}\left|\frac{\eta_z}{\eta}\right|^2 - \frac{1}{\eta^2}\left|\frac{\eta_w}{\eta}\right|^2\right) \nonumber\\
			&\  - \frac{2\gb}{\gl^2\eta}\Re(u_{\bz w}\gl_z\gl_{\bw})+ \frac{2}{\eta^3}\Re(u_{\bz w} \eta_z\eta_{\bw}) \nonumber\\
			&\ + \frac{2\gb}{\gl}\Re((u_{w\bz}u_{zz\bw} + \gl_w u_{z\bw})\frac{\eta_{\bz}}{\eta^2}) + \frac{2}{\eta}\Re((u_{z\bw}u_{ww\bz} - u_{w\bz}\eta_z)\frac{\eta_{\bw}}{\eta^2}).\nonumber
		\end{align}
		Combining (\ref{eq:Huzwovereta2}) with (\ref{eq:Hglflat}), and simplifying yields
		\begin{align}
			\HH(\gl + \frac{|u_{z \bw}|^2}{\eta}) =&\ -\gb\left|\frac{\gl_z}{\gl} + u_{z\bw}\frac{\gl_w}{\gl\eta}\right|^2 - \left|\frac{\gl_w}{\sqrt{\gl\eta}}-u_{w\bz}\frac{\eta_z}{\sqrt{\gl\eta^3}}\right|^2\label{eq:HWzz}\\
			&\ -\gb\left|\frac{u_{zz\bw}}{\sqrt{\gl\eta}} - u_{z\bw}\frac{\eta_z}{\sqrt{\lambda\eta^3}}\right|^2 - \left|\frac{u_{ww\bz}}{\eta} - u_{w\bz}\frac{\eta_w}{\eta^2}\right|^2.\nonumber
		\end{align}
		
		We note briefly that (\ref{eq:Hetam1}) and (\ref{eq:HWzz}) already imply that the diagonal entries of $W$ are subsolutions. All that remains is to show that the off-diagonal entries do not disturb the subsolution property. 
		
		To see that this is the case, we require $\HH u_{z\bw}/\eta$. We can compute this using the product rule and (\ref{eq:Hetam1}) and (\ref{eq:Huzwflat}) to obtain		
		\begin{align}
			\HH(\frac{u_{z\bar{w}}}{\eta}) =&\ -\gb\frac{\gl_z\gl_{\bw}}{\gl^2\eta} + \frac{\eta_z\eta_{\bw}}{\eta^3} + u_{z\bar{w}} \left(-\frac{\gb}{\eta^2}\left|\frac{\gl_w}{\gl}\right|^2 - \frac{2\gb}{\gl\eta}\left|\frac{\eta_z}{\eta}\right|^2 - \frac{1}{\eta^2}\left|\frac{\eta_w}{\eta}\right|^2\right)\label{eq:HWzw1}\\
			&\ + \frac{\gb}{\gl}(\gl_{\bar{w}}\frac{\eta_z}{\eta^2} + \frac{\eta_{\bar{z}}}{\eta^2} u_{zz \bar{w}}) + \frac{1}{\eta}(u_{z\bar{w}\bar{w}}\frac{\eta_w}{\eta^2} - \eta_z \frac{\eta_{\bar{w}}}{\eta^2}).\nonumber
		\end{align}
		It will be convenient to regroup the terms in (\ref{eq:HWzw}) as  
		\begin{align}
			\HH(\frac{u_{z\bar{w}}}{\eta}) =&\ \gb\frac{\eta_{\bar{z}}}{\sqrt{\gl\eta^3}} \left(\frac{u_{zz \bar{w}}}{\sqrt{\gl\eta}} - u_{z\bw}\frac{\eta_z}{\sqrt{\gl\eta^3}}\right) +  \left(\frac{u_{z\bar{w}\bar{w}}}{\eta} - u_{z\bw}\frac{\eta_{\bw}}{\eta^2}\right)\frac{\eta_w}{\eta^2} \label{eq:HWzw}\\
			&\ -\gb\left(\frac{\gl_z}{\gl} + u_{z\bar{w}} \frac{\gl_w}{\gl\eta}\right)\frac{\gl_{\bw}}{\gl \eta} + \gb\left(\frac{\gl_{\bar{w}}}{\sqrt{\gl\eta}} - u_{z\bar{w}} \left(\frac{\eta_{\bz}}{\sqrt{\gl\eta^3}} \right)\right)\frac{\eta_z}{\sqrt{\gl\eta^3}}\nonumber
		\end{align}
		
		Let $v\in \mathbb{C}^2$ with $v=\begin{pmatrix}
			a\\b
		\end{pmatrix}$ be arbitrary, then combining (\ref{eq:Hetam1}), (\ref{eq:HWzz}), and (\ref{eq:HWzw}) yields
		\begin{align*}
			\HH W(v,\bv)=&\ |a|^2 \left(-\gb\left|\frac{\gl_z}{\gl} + u_{z\bw}\frac{\gl_w}{\gl\eta}\right|^2 - \left|\frac{\gl_w}{\sqrt{\gl\eta}}-u_{w\bz}\frac{\eta_z}{\sqrt{\gl\eta^3}}\right|^2\right.\\
			&\ \left. -\gb\left|\frac{u_{zz\bw}}{\sqrt{\gl\eta}} - u_{z\bw}\frac{\eta_z}{\sqrt{\lambda\eta^3}}\right|^2 - \left|\frac{u_{ww\bz}}{\eta} - u_{w\bz}\frac{\eta_w}{\eta^2}\right|^2\right)\\
			&\ +2\Re\left(a\bb \left(\gb\frac{\eta_{\bar{z}}}{\sqrt{\gl\eta^3}} \left(\frac{u_{zz \bar{w}}}{\sqrt{\gl\eta}} - u_{z\bw}\frac{\eta_z}{\sqrt{\gl\eta^3}}\right) +  \left(\frac{u_{z\bar{w}\bar{w}}}{\eta} - u_{z\bw}\frac{\eta_{\bw}}{\eta^2}\right)\frac{\eta_w}{\eta^2}\right.\right.\\
			&\ \left.\left. -\gb\left(\frac{\gl_z}{\gl} + u_{z\bar{w}} \frac{\gl_w}{\gl\eta}\right)\frac{\gl_{\bw}}{\gl \eta} + \gb\left(\frac{\gl_{\bar{w}}}{\sqrt{\gl\eta}} - u_{z\bar{w}} \left(\frac{\eta_{\bz}}{\sqrt{\gl\eta^3}} \right)\right)\frac{\eta_z}{\sqrt{\gl\eta^3}}\right)\right)\\
			&\ + |b|^2\left( -\frac{\gb}{\eta^2}\left|\frac{\gl_w}{\gl}\right|^2 - \frac{2\gb}{\gl\eta}\left|\frac{\eta_z}{\eta}\right|^2 - \frac{1}{\eta^2}\left|\frac{\eta_w}{\eta}\right|^2\right).
		\end{align*}
		Which is obviously seen to be non-positive by simply performing the Cauchy-Schwarz and Young's inequalities indicated by the groupings. The proposition follows.
	\end{proof} 
\end{lemma}

\end{appendices}



\end{document}